\documentclass[11pt]{amsart}
\usepackage{fullpage}
\usepackage{amsfonts,amssymb,amsmath,amsthm}
\usepackage{hyperref}
\usepackage{color}
\usepackage{comment}

\theoremstyle{plain}
\newtheorem{theorem}{Theorem}[section]
 \newtheorem{corollary}[theorem]{Corollary}
 \newtheorem{lemma}[theorem]{Lemma}
 \newtheorem{proposition}[theorem]{Proposition}
 \theoremstyle{definition}
 \newtheorem{definition}[theorem]{Definition}
 
 \theoremstyle{remark}
 \newtheorem{remark}[theorem]{Remark}
 \numberwithin{equation}{section}

\def\eps{\varepsilon}
\def\R{{\mathbb R}}
\def\N{{\mathbb N}}
\mathchardef\mhyphen="2D 

\def \PullOne{\mathcal{I}_\Phi}
\def \PullTwo{\mathcal{U}_\Phi}

\def \bP {{\mathbb P}}
\def \bL {{\mathbb L}}
\def \Svd {{\Gamma_c\left(\mathcal{S}\left(\mathbb{G}M, |\Lambda|\mathcal{V}\right)\right)}}

\def\Tend#1#2{\mathop{\longrightarrow}\limits_{#1\rightarrow#2}}

\numberwithin{equation}{section}

\begin{document}
	\title[]{Some remarks on semi-classical analysis\\ on two-step Nilmanifolds}
	\author[C. Fermanian]{Clotilde~Fermanian~Kammerer}
	\address[C. Fermanian Kammerer]{
		Univ Paris Est Creteil, CNRS, LAMA, F-94010 Creteil, Univ Gustave Eiffel, LAMA, F-77447 Marne-la-VallÃ©e, France
	}
	\email{clotilde.fermanian@u-pec.fr}
	\author[V. Fischer]{V\'eronique Fischer}\address[V. Fischer]%
	{University of Bath, Department of Mathematical Sciences, Bath, BA2 7AY, UK} 
	\email{v.c.m.fischer@bath.ac.uk}
	\author[S. Flynn]{Steven Flynn}
	\address[S. Flynn]{University of Bath, Department of Mathematical Sciences, Bath, BA2 7AY, UK} 
	\email{spf34@bath.ac.uk}
	
	\begin{abstract} 	
		In this paper, we present recent results about the developement of a 
semiclassical approach in the setting of  nilpotent Lie groups and nilmanifolds.
 We focus on two-step  nilmanifolds and exhibit some properties of the  weak limits of sequence of densities associated with eigenfunctions of a sub-Laplacian. We emphasize  the influence of the geometry on these properties. 
	\end{abstract}

\maketitle

\makeatletter
\renewcommand\l@subsection{\@tocline{2}{0pt}{3pc}{5pc}{}}
\makeatother


\section{Introduction}

\subsection{Subelliptic operators and subelliptic estimates}

Sub-elliptic operators are an important class of operators containing  sub-Laplacians - also known as H\"ormander's sums of squares of vector fields~\cite{Ho} that generate the tangent space by iterated commutation. These operators also appear naturally in stochastic analysis as the Kolmogorov equations of stochastic ordinary differential equations are described in terms of second order differential operators which are often sub-Laplacians. In complex geometry,  Kohn Laplacian {(acting on functions)} on Cauchy-Riemann manifolds 
also gives an example of sub-elliptic operators. 
More generally, sub-elliptic operators appear in contact geometry, thereby having significant place.
\smallskip

One of their specific properties relies on the sub-elliptic estimates proved independently by Rothschild and Stein~\cite{RS} on the one hand, and Fefferman and Phong~\cite{FP}, on the other one. While, in the elliptic case, if $ \Delta u\in H^{s}(\R^d)$, then  $u\in H^{s+2}(\R^d)$, the gain of regularity is smaller for a sub-elliptic operator $\mathbb L= X_1^2 + \cdots + X_p^2$. Indeed, one then has 
 \[
 \mathbb L u\in H^{s}(\R^d)\;\Longrightarrow u\in H^{s+2/r}(\R^d)
 \]
 where $r$ is  the mean length to obtain spanning commutators. The Rothschild and Stein proof in ~\cite{RS} is based on Harmonic analysis on Lie groups, as developed in \cite{FS,RS}, via a lifting procedure consisting in the construction of a  nilpotent stratified Lie group for which the sub-elliptic operator is a sub-Laplacian.  It is in that spirit that we work here and we are interested in sublaplacians associated with a special type of manifolds called nilmanifolds,  that are naturally attached to a nilpotent  Lie group.

\subsection{Analysis on nilmanifolds}

In this paper, as is often the case in harmonic analysis, we restrict our attention to  nilpotent Lie groups that are stratified. We will further assume that their step is two later on. 

\subsubsection{Stratified Lie groups}
A stratified Lie group $G$ is a connected simply connected  Lie group whose (finite dimensional, real) Lie algebra $\mathfrak g$ admits an $\N$-stratification into linear subspaces, i.e.
$$
\mathfrak g = \mathfrak g_1 \oplus \mathfrak g_2 \oplus \ldots 
\quad\mbox{with} \quad [\mathfrak g_1,\mathfrak g_j]= \mathfrak g_{1+j}, \;\; 1\leq i\leq j.
$$
In this case, the group $G$ and its Lie algebra are nilpotent. Their step of nilpotency is the largest
number $s\in \N$ such that $\mathfrak g_s $ is not trivial.
In this paper, all the nilpotent Lie groups are assumed connected and simply connected.

\smallskip 

Once a basis $X_1,\ldots,X_{n}$
for~$\mathfrak g$ has been chosen, we may  identify 
the points $(x_{1},\ldots,x_n)\in \mathbb R^n$ 
 with the points  $x=\exp_G(x_{1}X_1+\cdots+x_n X_n)$ in~$G$ via the exponential mapping $\exp_G : \mathfrak g \to G$. 
 By choosing a basis adapted to the stratification, we derive 
  the product law   from the  Baker-Campbell-Hausdorff formula. We can also define 
the (topological vector) spaces $\mathcal C^\infty(G)$ and $\mathcal S(G)$  of smooth and  Schwartz functions on $G$ identified with $\R^n$.
This induces a Haar measure $dx$ on $G$ which is invariant under left and right translations and  defines Lebesgue spaces on~$G$, together with a (non-commutative) convolution for functions 
$f_1,f_2\in\mathcal S(G)$ or in~$L^2(G)$,
$$
 (f_1*f_2)(x):=\int_G f_1(y) f_2(y^{-1}x) dy,\;\; x\in G.
$$

The Lie algebra 
 $\mathfrak g$ is naturally  equipped with 
the family of dilations  $\{\delta_r, r>0\}$,  $\delta_r:\mathfrak g\to \mathfrak g$, defined by
$\delta_r X=r^\ell X$ for every $X\in \mathfrak g_\ell$, $\ell\in \N$
\cite{FS}.
 The associated group dilations derive from 
$$
\delta_r( \exp_G X)=\exp_G (\delta_r X), 
\quad r>0, \ X\in \mathfrak g.
$$
In a canonical way,  this leads to a notion of homogeneity for functions (measurable functions as well as distributions)  and    operators.
 For instance, the Haar measure is $Q$-homogeneous
where
$$
Q:=\sum_{\ell\in \mathbb N}\ell \dim \mathfrak g_\ell
$$
 is called the \emph{homogeneous dimension} of $G$.
Another example is obtained by 
identifying the elements of the Lie algebra $\mathfrak g$ with the left-invariant vector fields on $G$: we check readily that the elements of $\mathfrak g_j$ are homogeneous differential operators of degree $j$.

\smallskip 

When a scalar product is fixed on the first stratum $\mathfrak g_1$ of the Lie algebra $\mathfrak g$, the group $G$ is said to be Carnot. 
The intrinsic sub-Laplacian on $G$ is then the differential operator given by  
\[
\mathbb L_G := V_1^2 + \cdots + V_q^2,
\]
for any orthonormal basis   $V_1,\ldots, V_q$ of $\mathfrak g_1$.
{ We fix such a basis that will be used in different places of the paper.}

\subsubsection{Nilmanifolds}
A nilmanifold is the one-sided quotient of a nilpotent Lie group $G$ by a discrete subgroup $\Gamma$ of $G$.
In this paper, we will choose the left quotient of $G$ and denote it by  $M=\Gamma\backslash G$.
We will consider compact nilmanifolds, or equivalently cocompact subgroups~$\Gamma$. We denote by $x\mapsto \dot x$ 
the canonical projection which associates to $x\in G$ its class modulo $\Gamma$ in~$M$.

Recall that the Haar measure $dx$ on $G$ is unique up to a constant and, once it is fixed, $d\dot x$ is the only $G$-invariant measure on $M$ satisfying 
for any  function $f:G\to \mathbb C$, for instance continuous with compact support,
\begin{equation}
\label{eq_dxddotx}
	\int_G f(x) dx = \int_M \sum_{\gamma\in \Gamma} f(\gamma x) \ d\dot x.
\end{equation}
We may allow ourselves to write $dx$ for the measure on $M$ when the variable of integration is $x\in M$ and no confusion with the Haar measure is possible.

The canonical projection $G\to M$ induces a one-to-one correspondence between the set of functions on $M$   with the set of $\Gamma$-left periodic functions on $G$, that is, the set of functions $f$ on $G$ satisfying
$$
\forall x\in G,\;\;\forall \gamma\in \Gamma ,\;\; f(\gamma x)=f(x).
$$
With a function $f$ defined on $M$, we associate the $\Gamma$-left periodic function $f_G:x \mapsto f(\dot x)$ defined on $G$. Conversely, a
 $\Gamma$-left periodic function $f$ on $G$ naturally defines a function $f_M: \dot x \mapsto f(x)$ on $M$. 

Consider a linear continuous mapping $T:\mathcal S(G)\to \mathcal S'(G)$  
which is invariant under $\Gamma$ in the sense that
$$
\forall F\in \mathcal S(G), \;\; \forall \gamma\in \Gamma,
\;\; { T(F(\gamma\, \cdot)) = (TF)(\gamma\, \cdot)}.
$$ 
Then it  naturally induces \cite{F2022} an operator $T_M$ on $M$ via
$$
T_M f = (Tf_G)_M.
$$
Furthermore, $T_M:\mathcal D(M)\to \mathcal D'(M)$  is 
a linear continuous mapping. 
 Note that if $T$ is invariant under $G$, then it is invariant 
under $\Gamma$.
For instance, any left-invariant differential operator $T$ on $G$ induces a corresponding differential operator $T_M$ on~$M$.

\smallskip

Let us now assume that $G$ is a Carnot group.
The intrinsic sub-Laplacian on $M$ is 
 the operator $\mathbb L_M$ induced by $\mathbb L_G$  on $M$.
It is a differential operator that is essentially self-adjoint on $L^2(M)$; we will keep the same notation for its self-adjoint extension.
The spectrum of $-\mathbb L_M$  is a discrete and unbounded subset of $[0,+\infty)$.
Each eigenspace of $\mathbb L_M$ has finite dimension.
The constant functions on $M$ form the 0-eigenspace of $\mathbb L_M$, see e.g. \cite{F2022}.

\subsubsection{Objectives}

In this paper, we   consider   nilpotent Lie
 groups $G$ of step $s=2$ equipped with a scalar product. They are naturally stratified, (see Section \ref{subsubsec_step2G}) 
and so they will also be Carnot.
We will focus our attention on
sequences of eigenfuctions  $(\psi_k)_{k\in\N}$ and eigenvalues $(E_k)_{k\in\N}$ of $\mathbb L_M$, ordered in increasing order {and repeated according to multiplicity}: 
\begin{equation}\label{eq:eigen}
- \mathbb L_M \psi_k=E_k\psi_k,\qquad E_1\leq E_2\leq  \cdots \leq E_k\leq \cdots,\qquad E_k\Tend{k}{\infty} +\infty.
\end{equation}
We are interested in the measures on $M$ that are  limit points of the densities $|\psi_k(x)|^2 dx$ as $k$ tends to $+\infty$. Our result extends to operators
\[ -\mathbb L^{\mathbb U}_M= -\mathbb L_M +\mathbb U(x)
\]
where $x\mapsto \mathbb  U(x)$ is a smooth potential on $M$.
Our analysis will be using a semi-classical approach based on the harmonic analysis on the group $G$ in order to derive invariance properties of these measures.

   \subsection{Fourier analysis of step-two groups }
   Our semi-classical approach is based on the Fourier theory of the group, as developed in Harmonic analysis (see for example~\cite{FS,R+F_monograph}). In the rest of this paper, we will consider only a  nilpotent Lie group $G$ of step two
and its associated compact nilmanifolds
$M=\Gamma\backslash G$.

\subsubsection{Step-two groups}
\label{subsubsec_step2G}
As $G$ is step two, the derived algebra
$\mathfrak z:=[\mathfrak g, \mathfrak g]$ lies in the centre of $\mathfrak g$. Moreover,
denoting by $\mathfrak v$ a complement of $\mathfrak z$,  we have the decomposition:
\[
 \mathfrak g= \mathfrak v \oplus \mathfrak z.
\]
Note that $\mathfrak z =[\mathfrak v, \mathfrak v] $ and that this decomposition yields a stratification of $\mathfrak g$ with $\mathfrak g_1= \mathfrak v$, $\mathfrak g_2 =\mathfrak z$. Hence $G$ is naturally stratified with dilations given by 
 $\delta_\eps (V+Z)=\eps V +\eps^2 Z$ where 
 $\eps>0$, $V\in \mathfrak v$, $Z\in \mathfrak z$.
 Its topological dimension is $n= \dim \mathfrak v +\dim \mathfrak z$
while the homogeneous dimension is  
 $Q=\dim \mathfrak v + 2 \dim  \mathfrak z$. 
 We also assume that a scalar product has been fixed on $\mathfrak g$, and that $\mathfrak v$ is an orthogonal complement of $\mathfrak z$.

\subsubsection{The dual set}
\label{subsubsec_Ghat}
The dual set $\widehat G$ of $G$ is the set of  the equivalence classes of the irreducible unitary representations of $G$. We will often allow ourselves to identify a class of such representations with one of its representatives.
Since $G$ is a  nilpotent Lie group, its dual is the disjoint union of the (classes of unitary irreducible) representations of dimension one and of infinite dimension:
$$
\widehat G=\widehat G_1\sqcup \widehat G_\infty, 
\quad \widehat G_1:=
\{ \mbox{class\; of} \; \pi, \ \dim \pi=1 \}, \quad 
\widehat G_\infty:=
\{  \mbox{class\; of} \; \pi, \ \dim \pi=\infty \}.
$$
As $G$ is step two,  $\widehat G_1$ and $\widehat G_\infty$  can be described in a relatively simple manner.
\smallskip

(i) The (classes of unitary irreducible) one-dimensional representations are parametrized by the elements  $\omega\in \mathfrak v^*$ of the dual of $\mathfrak v$ and consists of the characters 
\[
 \pi^\omega (x)= {\rm e}^{i \omega(V)},\;\; x=\exp_G(V+Z), \;\;V\in \mathfrak v,\;\; Z\in\mathfrak z.
\]

(ii) The (classes of unitary irreducible) infinite  dimensional  representations are parametrised by a non-zero element $\lambda \in \mathfrak z^* \setminus\{0\} $ of the dual of $\mathfrak z$ and another parameter $\nu\in \mathfrak v^*$ which we now describe.  
  For any~$\lambda \in  \mathfrak z^\star$, we consider the skew-symmetric bilinear form on $\mathfrak v$ defined by 
\begin{equation}\label{skw}
\forall \, U,V \in  \mathfrak v \, , \quad B(\lambda) (U,V):= \lambda([U,V]) \, .
 \end{equation}
 We denote by  $\mathfrak r_\lambda$ the radical of $B(\lambda)$.
 The other parameter $\nu$ will be in the dual $\mathfrak r_\lambda^*$ of this radical.
 
 Using the scalar product on $\mathfrak g$,  we can construct the representation $\pi^{\lambda,\nu}$  for each $\lambda \in \mathfrak z^* \setminus\{0\}$ and $\nu \in \mathfrak r_\lambda^*$ as follows.
 First, we will allow ourselves to keep the same notation for the skew-symmetric form $B(\lambda)$ and the corresponding skew-symmetric linear map on $\mathfrak v$. Hence  $\mathfrak r_\lambda = \ker B(\lambda)$. As $ B(\lambda)$ is skew symmetric, we find an orthonormal basis of $\mathfrak v$
  $$\big (P_1^{\lambda} , \dots ,P_d^{\lambda},  Q_1^{\lambda} , \dots ,Q_d^{\lambda},R_1^{\lambda},\dots,R_k^{\lambda}\big )
  \qquad \mbox{with}\quad 
 k=k_\lambda := \dim \mathfrak r_\lambda , \quad 
 d=d_\lambda:=\frac{\dim \mathfrak v -k}2,
  $$ 
    where 
 the  matrix of~$B(\lambda)$ takes the block form
\begin{equation}
    \label{eq_matB}
    \begin{pmatrix}
  0_{d,d} & D(\lambda) & 0_{d,k}\\
  -D(\lambda)& 0_{d,d} & 0_{d,k}\\
  0_{k,d} & 0_{k,d} & 0_{k,k}
\end{pmatrix}.
\end{equation}
  Here $D(\lambda)$ is a diagonal matrix with positive diagonal entries depending on $\lambda$.
Note that 
$\mathfrak r_\lambda = \mbox{Span} \, \big (R_1^\lambda , \dots ,R_k^\lambda \big)$
and we decompose $ \mathfrak v$ as
$$
\mathfrak v = \mathfrak p_\lambda +  \mathfrak q_\lambda  +\mathfrak r_\lambda
\quad\mbox{where}\quad
\mathfrak p_\lambda:= \mbox{Span} \, \big (P_1^\lambda , \dots ,P_d^\lambda \big) \quad, \qquad \mathfrak q_\lambda:= \mbox{Span} \, \big (Q_1^\lambda , \dots ,Q_d^\lambda).
$$
 One may assume that the  above basis for $\mathfrak{v}$ depends continuously on $\lambda$.

The representation $\pi^{\lambda,\nu}$ acts on $L^2(\mathfrak p_\lambda)$ via
{
\begin{equation}\label{def:pilambdanu}
\pi^{\lambda,\nu}(x) \phi (\xi)
=e^{i\lambda\left(Z+ \left[ D(\lambda)^{\frac 12}\xi +\frac 12 P, Q\right]\right)}
e^{i\nu (R)}\phi\left(D(\lambda)^{\frac 12}\xi+P\right) ,
\quad \phi \in L^2(\mathfrak p_\lambda), \quad \xi\in \mathfrak p_\lambda,
\end{equation}
}
where $x$ is written as
$x = \exp_G (P+Q+R+Z)$ with $P\in \mathfrak p_\lambda$, $Q\in \mathfrak q_\lambda$, $R \in \mathfrak r_\lambda$, $Z\in \mathfrak z.$ If $\nu=0$, we will use the shorthand $\pi^{\lambda,0}=\pi^\lambda$.

With the representations described in (i) and (ii) above, the dual set of $G$ is:
$\widehat G=\widehat G_1 \sqcup \widehat G_\infty$ with 
$$
\widehat G_1=
\{ \mbox{class\; of} \; \pi^\omega, \,\omega\in  \mathfrak v^* \} \quad \mbox{and}\quad
\widehat G_\infty=
\{  \mbox{class\; of} \; \pi^{\lambda,\nu},\,\lambda\in\mathfrak z^*\setminus \{0\},\, \nu\in \mathfrak r_\lambda^*\}.
$$
This can be justified  in this case with the von Neumann theorem characterising the representations of the Heisenberg groups.
Equivalently, we can also use the orbit method which states that there is a one-to-one correspondence between $\pi\in \widehat G$ and the co-adjoint orbits $\mathfrak g^* /G$. 
The advantage of  the orbit method is that the Kirillov map $\mathfrak g^* /G \rightarrow \widehat G$ is a homeomorphism
\cite{brown},
giving us easy information on the topology of subsets of  $\widehat G$.
Furthermore,  one can  check that the co-adjoint action of $G$ on $\mathfrak g^* = \mathfrak v^* \oplus \mathfrak z^*$ leaves the $\mathfrak z^*$-component invariant. Hence,  we can  describe the co-adjoint orbit of any $\nu+\lambda \in \mathfrak g^* = \mathfrak v^* \oplus \mathfrak z^*$ by choosing
the unique representative as  the linear form $\omega=\nu$ if $\lambda=0$, and $\lambda+\nu$ with $\nu \in \mathfrak r_\lambda^* $ if $\lambda\not=0$. 
Via the Kirillov map, they correspond respectively to~$\pi^\omega$ and $\pi^{\lambda,\nu}$.

\subsubsection{The subsets $\Omega_k$ and $\Lambda_0$}
\label{subsubsec_Omegak}
As a set, $\mathfrak z^*\setminus \{0\}$ decomposes as the disjoint union of 
$$
\Omega_k:=\{\lambda \in \mathfrak z^*\setminus\{0\} \ : \ \dim \mathfrak r_\lambda =k\}, \quad k\in \N.
$$
Observe that 
$\Omega_k=\emptyset$  when $k>\dim \mathfrak v$ and also when $k=\dim \mathfrak v$ because if $k_{\lambda}=\dim \mathfrak v$ then $\mathfrak r_{\lambda}=\mathfrak v^{*}$, thus $B_{\lambda}=0$ and $\lambda=0$.  We denote by $k_0$ the smallest $k\in \N$ such that $\Omega_k\neq\emptyset$; roughly speaking, 
this is the set of $\lambda\in \mathfrak z^*$ for which $B(\lambda)$ is of smallest kernel. We  have 
\[
\mathfrak z^*\setminus \{0\} = \sqcup_{k_0\leq k < \dim \mathfrak v} \, \Omega_k.
\]

{We can describe $\cup_{k'\geq k}\Omega_{k'}$ as the set of $\lambda\in\mathfrak z^*\setminus \{0\} $ such that all the minors of $B(\lambda)$ (viewed as a matrix in the
basis that we have fixed) of order $\leq \dim \mathfrak v - k$ cancel, 
and $\Omega_k$ as the subset of $\cup_{k'\geq k}\Omega_{k'}$ formed by the $\lambda's$
such that at least one minor of order $=\dim \mathfrak v - k$
does not vanish.
Since $B(\lambda)$ is linear in $\lambda$,
 $\cup_{k'\geq k}\Omega_{k'}$ is an algebraic variety, and $\Omega_k$ is an open subset of it. Moreover, if $\Omega_k\neq \emptyset $ then $\cup_{k'> k}\Omega_{k'}$ is an algebraic subvariety with $\dim \cup_{k'> k}\Omega_{k'}<\dim \cup_{k'\geq k}\Omega_{k'}$.
 Consequently, 
$\Omega_k$ is an open subset of $\cup_{k'\geq k}\Omega_{k'}$
and it is either empty or dense in $\cup_{k'\geq k}\Omega_{k'}$.
}

{
We can decompose each $\Omega_k$ into further subsets, according to the multiplicity of the eigenvalues of
$B(\lambda)$ viewed as a matrix in a canonical basis.
Here, {we will be only considering} the case $k=k_0$
and denote by $\Lambda_0$ the set of $\lambda\in \Omega_{k_0}$ for which $B(\lambda)$ has the maximal number of distinct eigenvalues.
Recall that, by  the Cauchy residue formula, the multiplicity of a zero $z_0$ of a polynomial $p(z)$ is equal to $\oint_{|z -z_0| = r} \frac{p'(z)}{p(z)}dz  $ for $r$ small enough.
Applying this to $\det (B(\lambda)^2-z)$
in the case of maximal multiplicities implies that 
the multiplicities of the eigenvalues of $B(\lambda)^2$ for $\lambda\in \Lambda_0$ are locally constant and that the subset $\Lambda_0$ is open  in $\Omega_{k_0}$. Moreover, by the implicit function theorem, the eigenvalues of $B(\lambda)^2$ can be written locally as smooth  functions (even algebraic expressions)
of $\lambda\in \Lambda_0$.
Similar properties hold for each subset of $\Omega_{k_0}$ with fewer constraints on the multiplicities, implying that $\Lambda_0$ is dense in 
$\Omega_{k_0}$.
}

\smallskip

The Heisenberg groups correspond to the case when  $\dim \mathfrak z=1$ while 
the Heisenberg-type groups are exactly the  step-two nilpotent groups $G$ 
for which 
$B(\lambda)^2 = -|\lambda|^2 {\rm I}_{\mathfrak v}$.
Heisenberg-type groups and their nilmanifolds
have an H-type foliation as in~\cite{BGMR}, 
and so do the groups $G$ and their nilmanifolds 
when, more generally,  every $B(\lambda)$, $\lambda\in \mathfrak z^*\setminus\{0\}$, has a trivial radical $\mathfrak r_\lambda=\{0\}$.
Geometrically, these nilmanifolds are contact manifolds when the radicals are all trivial and $\dim\mathfrak z=1$,
and they are quasi-contact manifolds
 when the radicals may not be trivial. 
The analysis of the properties of weak limits of densities of eigenvalues of the sub-Laplacian  for contact manifolds was studied in~\cite{CdVHT} and for quasi-contact manifold of dimension four with radical  generically of dimension one was studied in \cite{Savale}. 

\smallskip

As the co-adjoint action is trivial on the $\mathfrak z^*$-component, 
the sets $\Omega_k$   may be viewed as the unions of the co-adjoint orbits of $\nu+\lambda \in \mathfrak g^* = \mathfrak v^* \oplus \mathfrak z^*$ with $\lambda \in \Omega_k$, or our chosen representatives for those co-orbits:
\begin{equation}
    \label{eq_Omegak_simlambdanu}
\Omega_k\sim 
\{(\lambda,\nu) \in \mathfrak z^* \times \mathfrak v^*, 
  \ \lambda\in \Omega_k,\, \nu\in \mathfrak r_\lambda^*\},
\end{equation}
and therefore identified via Kirillov's map with the following subset of $\widehat G$
$$
\Omega_k\sim 
\{\pi=\pi^{\lambda,\nu} \in \widehat G_\infty,  \ \lambda\in \Omega_k,\, \nu\in \mathfrak r_\lambda^*\}.
$$
We also proceed similarly for  $\Lambda_0$.
As subsets of $\widehat G_\infty$, they enjoy the same topological properties; for instance,
$\Omega_{k_0}$ which is an open dense subset of $\widehat G_\infty$.

\subsubsection{The Fourier transform} 
\label{subsubsec_FG}
Let $f\in L^1(G)$, the Fourier transform of $f$ is the field of operators   
$$
{\mathcal F} (f):=\{ \widehat f(\pi): {\mathcal H}_\pi \to {\mathcal H}_\pi, \pi\in \widehat G\}
\quad\mbox{given by}\
 \widehat f(\pi) 
 =\int_G f(x) \pi (x)^*dx, 
$$ 
for any (continuous unitary) representation $\pi$ of $G$.

The unitary dual $\widehat G$  is a standard Borel space, and there exists a unique positive Borel measure~$\mu$ on $\widehat G$ such that 
for any continuous function $f:G\to \mathbb C$ with compact support we have
 $$
 \int_G |f(x)|^2 dx =  \int_{\widehat G} \| \widehat f(\pi) \|_{HS({\mathcal H}_\pi)}^2 d\mu(\pi).
 $$
The measure $\mu$ is called the Plancherel measure and the formula above the Plancherel formula.
For instance, in the case of step-two groups, 
the Plancherel measure  is given  by 
$d\mu(\pi^{\lambda,\nu}) 
= c_0 \det (D(\lambda))\,   d\lambda d\nu,$
for a known constant $c_0>0$ \cite{CG,MuRi}; note that it  is supported on the subsets $\Omega_{k_0}$ or even $\Lambda_0$ of $\widehat G_\infty$ defined in Section~\ref{subsubsec_Omegak}.

The Plancherel formula extends the group Fourier transform  unitarily 
 to functions $f\in L^2(G)$:  their Fourier transforms are then a Hilbert-Schmidt fields of operators satisfying the Plancherel formula.
The group Fourier transform also extends readily to  classes of distributions, for instance the distributions with compact support and the  distributions whose associated right convolution operators are bounded on $L^2(G)$.
If $T$ is the associated operator, we denote by $\widehat T$ or $\pi(T)=\widehat T(\pi)$ the associated field of operators with
{
\[
\mathcal F (Tf) (\pi)= \pi(T)\circ  \mathcal F f(\pi) ,\;\;\forall f\in\mathcal S(G).
\]
}
In particular, the group Fourier transform extends to left-invariant differential operators. 

\smallskip

The considerations above are known for any nilpotent Lie group, and let  us consider the case of step-two groups.
The group Fourier transform  of $f\in L^1(G)$
gives  a scalar number at $\pi=\pi^\omega$ and  a bounded operator on 
$\mathcal H_{\pi^{\lambda,\nu}}= L^2( \mathfrak p^\lambda)$ for $\pi=\pi^{\lambda,\nu}$.
It is easy to compute that for the 1-dimensional representation, we have
$\pi^\omega(-\mathbb L_G)=|\omega|^2$. In the remainder of the paper, we will use the notation $\pi(\mathbb L)$
and $\widehat {\mathbb L} = \{ \pi(\mathbb L), \pi\in \widehat G\}$ and omit the index~$G$ in this context. 
 The case of representations of infinite dimension is more involved.
The following is known in great generality \cite{R+F_monograph}:
\begin{enumerate}
\item $\mathbb L_G$ and $\pi(\mathbb L)$ for $\pi\in \widehat G$  are essentially self-adjoint on $L^2(G)$ and $\mathcal H_\pi$;  we keep the same notation for their self-adjoint extensions. 
Hence they both admit spectral decompositions.
\item\label{item_sp} For each $\pi\in \widehat G\setminus \{1_{{\widehat G}}\}$, the spectrum ${\rm sp}(\pi(-\mathbb L))$ of 
 $\pi(-\mathbb L)$ is discrete and lies  in $(0,\infty)$ and each eigenspace is finite dimensional, while for $\pi=1_{\widehat G}$, $\pi(\mathbb L)=0$. 
\item Consider the spectral decomposition $\bP_\zeta$, $\zeta\geq 0$, of $-\mathbb L_G$, i.e. $-\mathbb L_G = \int_{0}^{\infty} \zeta d\bP_{\zeta}$.
For each $\pi\in \widehat G\setminus \{1_{{\widehat G}}\}$,
 the group Fourier transform $\pi( \bP_\zeta)$ of the projections $\bP_\zeta$ are  orthogonal projections of $\mathcal H_\pi$.
 Furthermore, they yield a spectral decomposition of $-\widehat {\mathbb L}$:
 $\pi(-\mathbb L)= \sum_{\zeta\in {\rm sp}(\pi(-\mathbb L)) } \zeta \pi(\bP_{\zeta})$.
\end{enumerate}
\smallskip 

In the step-two case, some of the properties above are easy to see. Indeed,  denoting by 
$$
\eta_j=\eta_j(\lambda), 1\leq j\leq d,
\quad\mbox{with the convention}\ 
0<\eta_1(\lambda) \leq \ldots \leq \eta_d(\lambda),
$$ 
the positive entries of 
$D(\lambda) = {\rm diag} (\eta_1,\ldots,\eta_d)$,
we readily compute 
{
\begin{equation}\label{eq:pi(P)}
    \pi^{\lambda,\nu}(P^\lambda_j)=\sqrt{\eta_j(\lambda)}  \partial_{\xi_j}\;\;\mbox{and}\;\;
    \pi^{\lambda,\nu}(Q^\lambda_j)=i\sqrt{\eta_j(\lambda)} \xi_j 
\end{equation}
and deduce from the additional observation 
$\pi^{\lambda,\nu}(R^\lambda_\ell)=i\nu_\ell$, $1\leq l\leq k$.
}
$$
\pi^{\lambda,\nu} (-\mathbb L)= H(\lambda)+|\nu|^2,
$$
where $H(\lambda)$ is the operator  on  $\mathcal H_\lambda$ given by 
\[
H(\lambda) =\sum_{1\leq j\leq d} \eta_j(\lambda) (-\partial_{\xi_j}^2 +\xi_j^2).
\]
which is 
{up to  multiplicative factors}   the harmonic oscillator of  $L^2(\R^d)$.
Recall that Hermite functions give 
an orthonormal basis of eigenfunctions of  $H(\lambda)$
with eigenvalues 
\begin{equation}
    \label{eq_zeta}
    \zeta(\alpha,\lambda)
:= \sum_{1\leq j\leq d} (2\alpha_j+1)\eta_j(\lambda), 
\qquad \alpha\in\N^d,
\end{equation}
see Section \ref{subsubsec_comput}.
Hence, 
the spectrum of $\pi^{\lambda,\nu} (-\mathbb L)$ is 
${\rm sp}(\pi^{\lambda,\nu} (-\mathbb L)) = \left\{ \zeta(\alpha,\lambda)+|\nu|^2,\;\alpha\in\N^d\right\},$
giving in this special case Property \eqref{item_sp} above.
Furthermore, the spectral projections $\pi^{\lambda,\nu}(\bP_\zeta)$  onto the eigenspaces of~$H(\lambda)$ are either zero or  orthogonal projections onto subspaces generated by Hermite functions. 

The properties above hold for any $\lambda\in \mathfrak z^* \setminus \{0\}$.
Restricting to $\Lambda_0$,  each $\eta_j(\lambda)$ is a smooth function of $\lambda\in\Lambda_0$ 
since the $\eta_j^2$'s are the eigenvalues of $B(\lambda)^2$ which are diagonalisable linear morphisms with eigenvalues of constant multiplicities depending smoothly on $\lambda$.
Therefore, $ \zeta(\alpha,\lambda)$ in \eqref{eq_zeta}  also depends smoothly on $\lambda$ in $\Lambda_0$.

\subsection{Main result}

Let $x\mapsto \mathbb U(x)$ be a smooth potential on $M$. Let $(\psi_k^{\mathbb U})_{k\in\N}$ be a sequence of eigenfunctions of $-\mathbb{L}_M^\mathbb{U}=-\mathbb L_M+\mathbb U$ according to
\begin{equation}\label{eq:eigenbis}
-\mathbb{L}_M^\mathbb{U}\psi_k^{\mathbb U}= E_k^{\mathbb U}\psi_k^{\mathbb U},\;\; k\in\N.
\end{equation}
{
Without loss of generality, we may assume $E_k^{\mathbb U}\geq 0$ for all $k\in\N$ (if not, we modify $\mathbb U$ by a constant).
}
Let $\varrho$ be a weak limit of  the density~$ | \psi_k^{\mathbb U}(x)|^2 dx$, then $\varrho$ decompose according to the structure of~$\widehat G$ and each of the elements of this decomposition enjoys its own invariances. These invariances {are expressed} in terms of the elements $\omega$, $\lambda$ and $\nu$ characterizing the points of $\widehat G$. 
    We will need the following notation to state  the result.
\begin{enumerate}
\item[(a)] 
For each $\lambda \in \mathfrak z^*$
and $\nu\in \mathfrak r_\lambda^*$, 
we  associate 
\[ 
\nu \cdot R^\lambda:= \nu_1 R^\lambda_1 + \cdots + \nu_k R^\lambda_k\ \in \mathfrak r_\lambda,
\]
where the $\nu_j$'s are the coordinates of $\nu$ in the dual of the orthonormal basis $(R_1^\lambda,\cdots, R_k^\lambda)$, i.e.
$\nu = \nu_1 (R_1^\lambda)^* + \ldots + \nu_k (R_k^{\lambda})^*$.
This definition is independent of the choice of the orthonormal basis $(R_1^\lambda,\cdots, R_k^\lambda)$  for $\mathfrak r_\lambda$.
\item[(b)] In the same spirit, 
for any $\omega\in \mathfrak v^*$, 
we associate 
$$
\omega \cdot V:= \omega_1 V_1 + \cdots \omega_q V_q \ \in \mathfrak v,
$$
where the $\omega_j$'s are the coordinates of $\omega$ in the dual of an orthonormal basis $(V_1,\cdots, V_q)$:
$\omega= \omega_1 V_1^* + \ldots + \omega_q V_q^*$.
Here, $q=\dim \mathfrak v$.
This definition is independent of the choice of the orthonormal basis $(V_1,\cdots, V_q)$  for $\mathfrak v$.
\item[(c)] If $k_0=0$
and $\lambda\in \Lambda_0$, each eigenvalue $\zeta=\zeta(\alpha,\lambda)$  in \eqref{eq_zeta} 
of $\pi^\lambda(\mathbb L)$
 depends smoothly on $\lambda$ in $\Lambda_0$.
The vector in $\mathfrak z$ corresponding to the gradient at $\lambda$ is denoted  by 
$$
\nabla_\lambda \zeta(\alpha,\lambda)=\nabla_\lambda \zeta \in \mathfrak z.
$$
\end{enumerate}

  \begin{theorem}[\cite{FF3,FL,FFF}]
  \label{theo:main}
 Let $(\psi_k^{\mathbb U})_{k\in\N}$ be a sequence of eigenfunctions of $-\mathbb{L}_M^\mathbb{U}=-\mathbb L_M+\mathbb U$ according to~\eqref{eq:eigenbis}. 
Then  a weak limit  $\varrho$  of  the density~$ | \psi_k^{\mathbb U}(x)|^2 dx$ decomposes as 
\begin{equation}\label{decomposition}
\varrho= \varrho^{\mathfrak v}+ \varrho^{\mathfrak z}
\end{equation}
with
\begin{enumerate}
\item
$\displaystyle{
\varrho^{\mathfrak v} (x) =
\int_{\mathfrak v^*} 
\varsigma \left(x,d\omega\right)
}$
where the measure $\varsigma$ is invariant by the flow
$$
(x,\omega) \mapsto ({\rm Exp} (s\,  \omega\cdot  V ) x,\omega),\;\; s\in\R
$$
 \item  $\displaystyle{
\varrho^{\mathfrak z} (x)= \sum_{k=0}^{\dim \mathfrak v -1}
\int_{(\lambda,\nu)\in \Omega_k}
\gamma_k(x,d\lambda,d\nu ) 
}$
with the identification \eqref{eq_Omegak_simlambdanu} for $\Omega_k$, 
with  each measure $\gamma_k(x,\lambda,\nu)$ being supported in $M\times \Omega_k$  where it is invariant under the flow given by
 $$
(x,(\lambda,\nu)) \longmapsto  ({\rm Exp} (s\, \nu \cdot R^{\lambda} ) x,(\lambda,\nu)),\quad s\in\R.
$$
\item 
Furthermore, in the case when $\Omega_0\neq \emptyset$, omitting $\nu=0$,
$$
\gamma_0(x,\lambda)
=\sum_{\alpha\in \N} \gamma_0^{(\alpha)}(x,\lambda),
$$
with each measure $1_{\lambda\in \Lambda_0}\gamma_0^{(\alpha)}$
being supported on $M\times \Lambda_0$ where it is invariant under the flow given by
{
$$
 (x,\lambda) \longmapsto  ({\rm Exp} (s\, \nabla_\lambda \zeta(\alpha,\lambda)) x,\lambda), \quad \ s\in\R.
$$
}
\end{enumerate}
\end{theorem}

In the case of the groups of Heisenberg type, $\eta_j(\lambda)=|\lambda|$ for all $j$, 
so 
$\Lambda_0=\Omega_0=\mathfrak z^*\setminus\{0\}$ and
\begin{equation}
    \label{eq_mathcalZlambda}
    \nabla_\lambda \zeta(\alpha,\lambda) = 
\mathcal Z^\lambda  \sum_{j=1}^{\dim \mathfrak v/2} (2\alpha_j+1), 
\qquad \mbox{where}\quad 
\mathcal Z^\lambda:=|\lambda|^{-1} \lambda^*,
\end{equation}
and $\lambda^*\in \mathfrak z$ corresponds to $\lambda$ 
by duality via the scalar product.
We therefore recover with Theorem~\ref{theo:main} the results of the first two authors in \cite{FF3}.
\smallskip

  Theorem \ref{theo:main} is a consequence of Theorem~\ref{theo:main2} below. It is  based on a microlocal approach and the measures $\gamma$ that appear in the statement above are microlocal objects that  can be compared with the semi-classical measures introduced in the 90s in the Euclidean context in~\cite{HMR,gerard_X,gerardleichtnam,GMMP}. The difference here is that the semi-classical calculus we use  is based on the Harmonic analysis of the group $G$ and on the Fourier transform introduced via representation theory as presented above. This setting has been introduced in~\cite{R+F_monograph}  in a microlocal context where no specific semi-classical scale $\eps$ is specified. It uses a pseudo-differential calculus with operator-valued symbols that can be composed with the  Fourier transform of the functions (that are also operator-valued).
\smallskip 

The construction of a pseudodifferential calculus on groups is an old question from the 1980s  \cite{Taylor84,rbeals2,bealsgreiner,christ_etc} that have known recent developments with an abstract point of view from the theory of algebra of operators  in~\cite{VanErp,VeY1,VeY2}, and with a  PDEs approach in~\cite{R+F_monograph,BFG1,FF1} with applications in control theory and observability~\cite{FL}
\smallskip 

We conclude this section with some comments about Theorem~\ref{theo:main}.  It is noticeable that there is coexistence of two kinds of behaviour, with a splitting of the measure $\gamma$ corresponding to the different types of elements of $\widehat G$. 
In the context of the Heisenberg group, $\Lambda_0=\Omega_0\neq \emptyset$ and $\nabla_\lambda\zeta $ is colinear to $
\mathcal Z^\lambda$ (see \eqref{eq_mathcalZlambda})
and this is linked to the wave aspect of the sub-Laplacian in this group  pointed out in~\cite{BGX,BS,CdVHT,BFG2}. 
On other nilpotent Lie groups where $\Omega_0=\emptyset$, the other vector fields involved, 
$\nu\cdot R^\lambda$,
are more of Schr\"odinger's type.

\medskip 

\noindent{\bf Acknowledgements}. 
The authors acknowledges the support of the Leverhulme Trust via Research Project Grant 2020-037.
The first author thank Cyril Letrouit for inspiring discussions.

\section{ Noncommutative semi-classical setting }

\subsection{ Semi-classical pseudodifferential operators}
\label{subsec_pseudodiff}
We consider the set {$\mathcal A_0$} of fields of operators $\{\sigma(x,\pi)\in{\mathcal L}({\mathcal H}_\pi), \; (x,\pi)\in M\times \widehat G\}$ such that 
$$
\sigma(x,\pi)=\mathcal F \kappa_x (\pi) = \int_G \kappa_x(z) \pi(z)^* dz, 
$$
where $x\mapsto \kappa_x(\cdot)$ is in ${\mathcal C}^\infty (M,\mathcal S(G))$. 
{We call the function $\kappa_x$ the {\it convolution kernel} associated with the symbol  $\sigma$. In the spirit of the works~\cite{BFG2,R+F_monograph}}, and when $\eps\ll 1$ is a semi-classical parameter, the $\eps$-quantization of the  symbols  $\sigma\in \mathcal A_0$ is given by 
\begin{align*}
{{\rm Op}_\eps(\sigma)}f(x)& = \int_{\hat G}{\rm tr} \left( \pi(x) \sigma(x,\eps\cdot  \pi)  \widehat f(\pi) \right) d\mu(\pi), \;\; f\in\mathcal S(M),\;\; x\in M.
\end{align*}
{Here, $\eps\cdot\pi$ denotes the class in $\widehat G$ of the irreducible representation $x\mapsto\pi(\delta_\eps x)$.}
Setting 
\[
\kappa^\eps_x(z)= \eps^{-Q} \kappa_x(\delta_{\eps^{-1}} z),
\]
the $\eps$-quantization then obeys to 
\begin{equation}\label{eq:kernel}
{{\rm Op}_\eps(\sigma)}f(x)=  \int_{ G} \kappa_x^\eps(y^{-1} x) f(y)dy= \sum_{\gamma\in\Gamma} \int_{y\in M} \kappa^\eps_x(\gamma y^{-1} x) dy,  \;\; f\in\mathcal S(M),\;\; x\in M.
\end{equation}
As in the case of groups (see~\cite{FF2}), the family $\left({\rm Op}_\eps(\sigma)\right)_{\eps>0}$ is a bounded family in $\mathcal L(L^2(M))$: 

 \begin{proposition} 
There exists $ C>0$ such that for all $\sigma\in {\mathcal A}_0$ and $\eps >0$,
$$\| {\rm Op}_\eps(\sigma)\|_{{\mathcal L}(L^2(M))}\leq
  \int_G \sup_{x\in M} |\kappa_x (z)|dz.$$
\end{proposition}

\begin{proof} By Young's convolution inequality
$$\| f*  \kappa^\eps_{x}(\gamma\cdot) \|_{L^2(M)}\leq \| \sup_{x\in M}| \kappa^\eps_{x}(\gamma\cdot) | \|_{L^1(M)} \|  f\|_{L^2(M)} ,$$
$${\rm with}\;\; \| \sup_{x\in M} | \kappa^\eps_{x}(\gamma\cdot)|  \|_{L^1(M)}= \eps^{-Q}\int_M \sup_{x\in M} |  \kappa_{x}(\eps^{-1}\cdot \gamma y) |dy =\int_{\gamma^{-1}M}\sup_{x\in M} |\kappa_x(y)|dy.$$
Therefore, using~\eqref{eq:kernel}, we deduce 
\begin{align*}
\| {{\rm Op}_\eps(\sigma)}f\|_{L^2(M)} &\leq \sum_{\gamma\in \Gamma}\| f*  \kappa^\eps_{x}(\gamma\cdot) \|_{L^2(M)}\\
&\leq \|  f\|_{L^2(M)}\sum_{\gamma\in \Gamma}\int_{\gamma^{-1}M}\sup_{x\in M} |\kappa_x(y)|dy
= \|  f\|_{L^2(M)} \int_{G}\sup_{x\in M} |\kappa_x(y)|dy.
\end{align*}
  \end{proof}
  
  Besides, this semi-classical pseudodifferential calculus {enjoys} symbolic calculus (see Proposition~3.6 in~\cite{FF2} in the case of groups and Proposition~2.2 in~\cite{FL} for the extension to nilmanifolds). 

\subsection{Semi-classical measures}
Let us first introduce our notion of operator-valued measures introduced in the earlier papers of the first two authors.
We will use the same notation as in those paper, even if it means  using the Greek letter $\Gamma$ for the trace-class operators $\Gamma(x,\pi)$.
We think that there is no possible confusion with our current notation for the co-compact discrete subgroup~$\Gamma$ of~$G$, and thus will allow this small conflict of notation

\smallskip

We consider  pairs $(\Gamma,\gamma)$ consisting in a positive Radon measure~$\gamma$ on $M\times \widehat G$ and a measurable field over $(x,\pi) \in M\times \widehat G$ of trace-class operators $\Gamma(x,\pi)$ on ${\mathcal H}_\pi$
satisfying
$$
\int_{M\times \widehat G} {\rm Tr}\left| \Gamma (x,\pi)\right| d\gamma(x,\pi)<\infty.
$$
We equip the set of such pairs with the equivalence relation $(\Gamma,\gamma)\sim (\Gamma',\gamma')$ given by the existence of a measurable function $f: M\times \widehat G\to \mathbb C$ such that 
\[
\gamma'=f\gamma\;\;\mbox{ and}\;\;\Gamma'=f^{-1} \Gamma,\;\; \gamma-a.e.
\]
We denote by ${\mathcal M}_{ov}(M\times \widehat G)$ the set of equivalence classes for this relation and by $\Gamma d\gamma$ the class of the  pair $(\Gamma, \gamma)$.
If $\Gamma \geq 0$, then we say that the operator valued measure $\Gamma d\gamma$ is positive, and we denote by 
$\mathcal M^+_{ov}(M\times \widehat G)$ the set of the positive operator-valued measures  
on $M\times \widehat G$. 
They  characterize bounded families in $L^2(M)$ according to the following theorem. 
 
\begin{proposition}[\cite{FF2,FF3}]\label{prop:scm}
 Let $(\psi^\eps)_{\eps>0}$ be a bounded family in $L^2(M)$.  
   There exist a subsequence
   $\eps_k\to 0$ as $k\to \infty$,
and an operator-valued measure
{$ \Gamma d\gamma \in  {\mathcal M}_{ov}^+(M\times \widehat G)$} satisfying
$$
\forall \sigma\in {\mathcal A}_0,\qquad   
\left({\rm Op}_{\eps_k} (\sigma) \psi^{\eps_k},\psi^{\eps_k}\right)
\underset{k\to \infty} {\longrightarrow}
\int_{ M\times \widehat G} {\rm Tr}\left(\sigma(x,\pi) \Gamma(x,\pi)\right)d\gamma(x,\pi) .
$$
\end{proposition}

Continuing with the setting of the statement above, we say then that
the operator-valued measure $\Gamma d\gamma$ is a \emph{semi-classical measure} of $(\psi^\eps)_{\eps>0}$ at the scale $\eps$. 
A given family $(\psi^\eps)_{\eps>0}$ may have several semi-classical measures, depending on different subsequences $(\eps_k)_{k\in\N}$. The knowledge of all these families indicates the obstruction to strong convergence  in $L^2(M)$ of the family $(\psi^\eps)_{\eps>0}$. 

\smallskip 

The scale $\eps$ is particularly interesting for analyzing the oscillations of a family $(\psi^\eps)_{\eps>0}$ that satisfies weighted Sobolev estimates such as 
\begin{equation}\label{eps-osc}
\exists s,C>0,\;\;\forall \eps>0,\;\;  \| (-\eps^2 \mathbb L_M)^{s\over 2} \psi^\eps \|_{L^2(M)} \leq C.
\end{equation}
Indeed, one can then link the weak limits of the energy densities with the semi-classical measures:
 
\begin{proposition}[\cite{FF3}]
Assume $(\psi^{\eps})_{\eps>0}$ satisfies~\eqref{eps-osc} and that $\Gamma d\gamma$ is a semi-classical measure of $(\psi^{\eps})_{\eps>0}$ for the subsequence $(\eps_k)_{k\in\N}$.
Then  for all $\phi\in{\mathcal C}^\infty (M)$,
\begin{equation}\label{cor:eps-osc}
\limsup_{k\rightarrow +\infty}  \int_{M} \phi(x) |\psi^{\eps_k}(t,x)|^2 dx
=
{
\int_{ M\times \widehat G}
}
\phi(x)  {\rm Tr} \left(\Gamma(x,\pi) \right) d\gamma(x,\pi).
\end{equation}
\end{proposition}

\subsection{ Application to quantum limits}

Let us now come back to the sequence \eqref{eq:eigen} of eigenfunctions $(\psi_k^{\mathbb U})_{k\in\N}$ of the  sub-Laplacian  operator $-\mathbb L_M^{\mathbb U}=-\mathbb L_M+\mathbb U(x)$ for a compact nilmanifod $M=\Gamma\backslash G$ whose underlying group $G$ is step two. Denoting by $E_k^{\mathbb U}$ the associated sequence of eignevalue; we set  
\[ \eps_k=(E_k^{\mathbb U})^{-1/2}\]
we obtain a semi-classical scale such that the sequence $(\psi_k^{\mathbb U})_{k\in\N}$ is $\eps_k$-oscillating. Thus any weak limit~$\varrho$ of the energy density $|\psi_k^{\mathbb U}(x)|^2 dx$ is the marginal of a semi-classical measure $\Gamma d\gamma$ of the family $(\psi_k^{\mathbb U})_{k\in\N}$ according to~\eqref{cor:eps-osc}. Therefore, the properties of the semi-classical measures of the sequence $(\psi_k^{\mathbb U})_{k\in\N}$ will reflect on any weak limit of the energy density.

\smallskip

We now omit the index $k\in\N$ and  focus on the  semi-classical measures of a family of normalized functions $(\psi^\eps)_{\eps>0}$ that satisfy 
\begin{equation}\label{pb}
-\eps^2\mathbb{L}_M^\mathbb{U} \psi^\eps=\psi^\eps,
\end{equation} 
where $\mathbb U \in \mathcal C^\infty(M)$ is a potential on $M$.

As $G$ is a nilpotent Lie group, 
the elements of $\mathcal M^+_{ov} (M\times \widehat G)$ split into two parts 
$$
\Gamma d\gamma
\ =\ {\bf 1}_{M\times \widehat G_1} \Gamma d\gamma \ + \ {\bf 1}_{M\times \widehat G_\infty}\Gamma d\gamma 
$$
In particular, on $M\times \widehat G_1$, we may assume  $\Gamma=1$, while on $M\times \widehat G_\infty$, the  trace-class  operator $\Gamma(x,\pi^{\lambda,\nu})$ acts  on $\mathcal H_{\pi^{\lambda,\nu}}= L^2(\mathfrak p_\lambda)$ in the case of a step-two group $G$.
\smallskip

We can already  observe that the decomposition~\eqref{decomposition} in Theorem \ref{theo:main} is due to the split above:  
 the measure $\varrho^{\mathfrak v}$ is the restriction of $\Gamma d\gamma$ to $M\times \widehat G_1$, 
 while the restriction to $M\times \widehat G_\infty$ yields a more involved measure $\varrho^{\mathfrak z}$. {The invariance then comes from} the theorem below. In this statement,
we will allow ourselves to use  the identifications (see Sections \ref{subsubsec_Ghat}
 and \ref{subsubsec_Omegak}):
$$
  \widehat G_1\sim \mathfrak v^*
\qquad\mbox{and}
\qquad
\widehat G_\infty \sim \sqcup_{k=k_0}^{\dim \mathfrak v-1}
\Omega_{k}.
$$

  \begin{theorem}\label{theo:main2}
  Let $(\psi^\eps)_{\eps>0}$ be a  family of normalized functions satisfying~\eqref{pb} and $\Gamma d\gamma$ one of its semi-classical measures. Then we have the following properties:
\begin{itemize}
\item[(i)]  Localization: 
\[ 
\pi(\mathbb L) \Gamma(x,\pi) =  \Gamma(x,\pi)  \pi(\mathbb L)= -\Gamma(x,\pi),\qquad
\gamma(x,\pi)\, a. e.
\]
which implies 
\begin{enumerate}
\item The scalar measure ${\bf 1}_{M\times \widehat G_1}\gamma  $ on $M\times \widehat G_1$ is supported in 
$\{(x,\pi^{\omega})\in M\times \widehat G_1 , |\omega|=1\}$.
\item 
 Setting
$\Gamma_{\zeta}:={\bf 1}_{M\times \widehat G_\infty} \widehat \bP_{\zeta}\Gamma$ for each $\zeta>0$, we have 
$\Gamma(x,\pi) = \sum_{\zeta\in {\rm sp} (\pi(-\mathbb L))}
\Gamma_{\zeta}(x,\pi)
$ for $\gamma$-almost every $(x,\pi)\in M\times \widehat G_\infty$.
Moreover, it satisfies $\zeta \Gamma_{\zeta}d\gamma  =  \Gamma_{\zeta}d\gamma$ in $\mathcal M_{ov}^+(M\times \widehat G)$. In other words, 
 $\zeta =1$ on the support of the measure ${\rm Tr}(\Gamma_\zeta(x,\pi)) \gamma(x,\pi)$.
 \end{enumerate}
\item[(ii)] {Invariance}:
\begin{enumerate}
\item The scalar measure ${\bf 1}_{M\times \widehat G_1}\gamma  $
is invariant under the flow 
$$
(x,\pi^\omega)\longmapsto ({\rm Exp}(s\omega\cdot V) x,\pi^\omega), 
\quad s\in\R .
$$
\item 
\begin{enumerate}
\item For each $\zeta>0$, 
the operator valued measure $\Gamma_{\zeta}d\gamma={\bf 1}_{M\times \widehat G_\infty} \widehat \bP_{\zeta}\Gamma d\gamma$  
is supported in $M\times \widehat G_{\infty}$  
where it
 is
invariant under  the flow
$$
(x,\pi^{\lambda,\nu})\longmapsto ({\rm Exp}(s\nu\cdot R^\lambda) x,\pi^{\lambda,\nu}), 
\quad s\in\R .
$$
\item Assume $\Omega_{0}\neq \emptyset$. 
For each $\zeta>0$ {parametrized smoothly by $\lambda$}, 
the  operator valued measure ${\bf 1}_{M\times \Lambda_{0}}\Gamma_{\zeta}  d\gamma $  is supported on $M\times \Lambda_0$  where it is invariant under  the  flow
{
$$
(x,\pi^{\lambda})\longmapsto 
({\rm Exp}(s \nabla_\lambda \zeta ) x,\pi^\lambda), 
\quad s\in\R.
$$
}
\end{enumerate}
\end{enumerate}
\end{itemize}
\end{theorem}

Note that the flow invariances may be different for various $\zeta$ in Part (2) (b). This was already observed on the groups of Heisenberg type where  $\Omega_0=\Lambda_0=\mathfrak z^*\setminus \{0\} \sim \widehat G_\infty$ (see~\cite{FF3,FL}).
The invariance of Point (2)(a) is empty in that case since the flow map of (2)(a) reduces to identity on $\Omega_0$.
\smallskip

 Theorem~\ref{theo:main2}   implies Theorem~\ref{theo:main} through the identification that has been mentioned above:
 $$
 \varrho^{\mathfrak v} (x)= \int_{\omega\in \mathfrak v^*} d\gamma(x,\pi^\omega)
 \qquad\mbox{and}\qquad
 \varrho^{\mathfrak z} (x)=  \int_{\pi\in \widehat G_\infty}
 {\rm Tr}(\Gamma (x,\pi) ) d\gamma(x,\pi). 
 $$

\subsection{Main ideas of the proof}

 Theorem~\ref{theo:main2} is inspired by the results~\cite{FF2,FL} where the group~$G$ was assumed to be of Heisenberg type.
 We follow here the ideas developed in these papers and extend them to general  two-step groups.
 We explain below the main elements of the proof that rely on technical lemmata that are discussed in Section~\ref{sec:tech}. 
 
One can notice that, formally, 
 \begin{equation}\label{eq:LU}
 -\eps^2\mathbb{L}_M^\mathbb{U}= -{\rm Op}_\eps (\pi(\mathbb L))+\eps^2{\rm Op}_\eps(\mathbb U),
 \end{equation}
 which implies that the term involving the potential~$\mathbb U$ is of lower order than the operator $\eps^2\mathbb L_M$ itself.
{For both the proof of the localisation results and the invariance ones, we start from some relations coming from the $(\psi^\eps)_{\eps>0}$ being eigenfunctions of the subLaplacian.}  We then use symbolic calculus as developed in~\cite{FF2,FL}
 {to analyse  these algebraic relations and compute precisely the symbols involved in the calculus. Finally, passing to the limit $\eps\rightarrow 0$, we investigate what the resulting equations mean for the semi-classical measure. We restrict ourselves to the zone  $\widehat G_1$ or $\widehat G_\infty$ by using symbol belonging to the von Neumann algebra generated by $\mathcal A_0$. Another important ingredient of the proof consists in analyzing the different behavior of symbols that commute with $\widehat{\mathbb L}$ and those who don't. These technical points are developed in Section~\ref{sec:tech}.    }
 \medskip
 
(i) {\it Localization}. Let $\sigma\in\mathcal A_0$. By the definition of the family $(\psi^\eps)_{\eps>0}$, we have (by equation \eqref{pb})
\[
\left({\rm Op}_\eps(\sigma) (-\eps^2 \mathbb{L}_M^\mathbb{U}\psi^\eps),\psi^\eps\right)_{L^2(M)} = 
\left({\rm Op}_\eps(\sigma) \psi^\eps,-\eps^2 \mathbb{L}_M^\mathbb{U}\psi^\eps\right)_{L^2(M)}=  
\left({\rm Op}_\eps(\sigma) \psi^\eps,\psi^\eps\right)_{L^2(M)}.
\]
{By passing to the limit and using~\eqref{eq:LU}, the definition of the semi-classical measures as in Proposition~\ref{prop:scm} and } the properties of the calculus \cite{FF2,FL},
give that any semi-classical measure $\Gamma d\gamma$ of $(\psi^\eps)_{\eps>0}$ satisfies
\begin{align}
\label{eq_pfi}
\int_{M\times \widehat G}
{\rm Tr} \left(\sigma(x,\pi) \pi(\mathbb L)\Gamma (x,\pi) \right) d\gamma(x,\pi)
&=
\int_{M\times \widehat G}
{\rm Tr} \left(\pi(\mathbb L)\sigma(x,\pi) \Gamma (x,\pi) \right) d\gamma(x,\pi)
\\
&=-
\int_{M\times \widehat G}
{\rm Tr} \left(\sigma(x,\pi) \Gamma (x,\pi) \right) d\gamma(x,\pi).
\nonumber
\end{align}
This readily implies the first localization property in (i).
The rest  of (i) follows as we can now apply \eqref{eq_pfi}
not only to symbols $\sigma$ in $\mathcal A_0$, 
but also in the von Neumann algebra generated by $\mathcal A_0$, in particular to ${\bf 1}_{M\times \widehat G_1}\sigma$ and to ${\bf 1}_{M\times \widehat G_\infty}\sigma$, see Lemma \ref{lem_vNstuff}.
Furthermore, \eqref{eq_pfi} shows the commutation of $\Gamma $  with $\widehat {\mathbb L}$ so also with the {spectral projectors} $\widehat {\mathbb P}_\zeta$ for $\zeta>0$.
Therefore, with the notation of Section \ref{subsec_B}, 
our classical measure $\Gamma d\gamma$ is in $\mathcal M_{ov}(M\times \widehat G)^{(\widehat {\mathbb L})}$, {the subspace of semi-classical measures that commute with $\widehat {\mathbb L}$}. 
Hence, by the analysis in Section \ref{subsec_B}, we only need to consider symbols $\sigma$ 
in $\mathcal B_{0}$ which is the space of the symbols in $\mathcal A_{0}$ that commutes with $\widehat{ \mathbb L}.$

\medskip

(ii) {\it Invariance}. We now take advantage of the fact that for all $\sigma\in\mathcal A_0$,
\begin{equation}\label{zorro19}
\left([{\rm Op}_\eps(\sigma) ,- \eps^2\mathbb{L}_M^\mathbb{U}]\psi^\eps,\psi^\eps\right)_{L^2(M)} = 0.
\end{equation}
Setting $\pi(V)\cdot V:= \sum_{j=1}^q\pi(V_j)V_j$ for any orthonormal basis of $V_1,\ldots, V_q$ of $\mathfrak v$, a computation gives
for
$\sigma\in\mathcal A_0$,
\begin{align}\label{A2}
\frac 1\eps [{\rm Op}_\eps(\sigma) ,- \eps^2\mathbb{L}_M^\mathbb{U} ]= 
-\frac 1\eps {\rm Op}_\eps([\sigma, \pi({ \mathbb L} )])
+
2\, {\rm Op}_\eps(\pi(V)\cdot V\sigma) + \eps\, {\rm Op}_\eps(\mathbb L \sigma)+ \eps [{\rm Op}_\eps(\sigma) ,\mathbb U(x) ].
\end{align}
For symbols $\sigma \in \mathcal B_{0}$ {(which then commute with $\widehat{\mathbb L}$)}, 
 the term in $\frac 1\eps$ in the right-hand side vanishes and 
we deduce by passing to the limit that  any semi-classical measure $\Gamma d\gamma$ of $(\psi^\eps)_{\eps>0}$ satisfies  
\begin{equation}\label{A}
\forall \sigma\in\mathcal B_0,\qquad
 \int_{M\times \widehat G}{\rm Tr} \left( \pi(V)\cdot  V \sigma(x,\pi) \Gamma(x,\pi) \right) d\gamma(x,\pi) =0.
\end{equation}

{Let us prove Part (2)(a). As for Part (1),}
we can apply this to the elements 
${\bf 1}_{M\times \widehat G_1}\sigma$ and to ${\bf 1}_{M\times \widehat G_\infty}\sigma$
of the von Neumann algebra generated by $\mathcal B_0$, see Lemma \ref{lem_vNstuff}.
 We obtain first that \eqref{A} holds with integration over $M\times \widehat G_1$;
Part (ii)(1) then follows from this and
 Corollary \ref{cor_prop_Bdual}.
Then, we obtain that \eqref{A} holds with integration  on $M\times \widehat G_\infty$. 
This yields 
\begin{align}
\label{Ainfty}
0&= \int_{M\times \widehat G_\infty}{\rm Tr} \left( \pi(V)\cdot  V \sigma(x,\pi) \Gamma(x,\pi) \right)  d\gamma(x,\pi) 
\\& =
  \int_{M\times \widehat G_\infty}\sum_{\zeta\in {\rm sp}(\pi(\mathbb L) }
  {\rm Tr} \left( \pi(\bP_{\zeta})(\pi(V)\cdot  V) \pi(\bP_{\zeta})\ \sigma(x,\pi)  \Gamma(x,\pi) \right)  d\gamma(x,\pi) ,\nonumber
 \end{align}
since  $\sum_{\zeta\in {\rm sp}(\pi(\mathbb L) }\pi(\bP_{\zeta}) $ is the identity operator on $\mathcal H_{\pi}$ and $\pi(\bP_{\zeta}) = \pi(\bP_{\zeta})^{2}$ commutes with $\sigma(x,\pi)$ and $\Gamma(x,\pi)$.
Furthermore, for $\pi=\pi^{\lambda,\nu}$, 
it follows from Section~\ref{sec:tech} {(see~\eqref{non_com_PQ})}
$$
\forall \zeta>0,\qquad 
\pi(\bP_{\zeta})(\pi(P^\lambda) \cdot P^\lambda)\pi(\bP_{\zeta}) =0,
\qquad
\pi(\bP_{\zeta})(\pi(Q^\lambda) \cdot Q^\lambda) \pi(\bP_{\zeta})=0.
$$
Hence \eqref{Ainfty} becomes 
\[
\forall \sigma\in\mathcal B_0,\qquad
\int_{M\times \widehat G_{\infty}} {\rm Tr} \left( \nu\cdot R^\lambda \sigma(x,\pi^{\lambda,\nu}) \Gamma(x,\pi^{\lambda,\nu}) \right)  d\gamma(x,\pi^{\lambda,\nu}) 
=0.
\]
This  implies Part  (ii)(2)(a) by Proposition \ref{prop_Bdual}
as $\Gamma d\gamma $ { is in the set  $\mathcal M_{ov}(M\times \widehat G)^{(\widehat {\mathbb L})}$ of operator-valued measures that commute with $\widehat{\mathbb L}$}.

\smallskip 

{Let us prove Part (2)(b). We now assume $\Omega_{0}\neq\emptyset$. Indeed,} on $\Omega_{0}$,  the  analysis above  does not yield anything  since $\nu\cdot R^{\lambda}=0$  on $\Omega_{0}$.
We will need the following observation:
\begin{lemma}
\label{lem_sigmaeta}
	If $\sigma\in \mathcal A_0$ and $\eta\in \mathcal S(\mathfrak z^*)$, 
	then 
	the symbol $\sigma \eta$ given by $(\sigma\eta)(x,\pi^{\lambda,\nu})=\sigma(x,\pi^{\lambda,\nu}) \eta(\lambda)$ is in $\mathcal A_0$.
	If $\sigma\in \mathcal B_0$ then $\sigma \eta\in \mathcal B_0$. 
\end{lemma}
\begin{proof}
	If  $\kappa_{x}(y)$ is the kernel of $\sigma$, then we check readily that $(y_{\mathfrak v},y_{\mathfrak z})\mapsto (\kappa_{x}(y_{\mathfrak v},\cdot)*_{\mathfrak z} \mathcal F_{\mathfrak z}^{-1} \eta) (y_\mathfrak z)$ is the kernel of $\sigma \eta$.
	The rest follows.
\end{proof}

By Lemma \ref{lem_sigmaeta}, 
if $\sigma_1\in \mathcal B_0$ and 
if  $\eta\in \mathcal S(\mathfrak z^*)$ is supported in the dense open subset $\Omega_{0}$ of $\mathfrak z^*\setminus\{0\}$, then $\sigma:=\sigma_{1} \eta$ is supported in $M\times \Omega_{0}$. 
Moreover, by Lemma~\ref{lem:computation}, there exists a symbol $T\sigma \in\mathcal A_0$ such that 
\[
\pi(V)\cdot V \sigma = [ T\sigma,\pi(-\mathbb L)].
\]
Therefore, using the additional fact 
\[
[{\rm Op}_\eps(T\sigma), \mathbb U(x)]= O(\eps) \;\;\mbox{in}\;\;\mathcal L(L^2(M)),
\] 
{
the equation \eqref{A2} gives}
\begin{align*}
&\frac 1\eps \left( {\rm Op}_\eps(\pi(V)\cdot V \sigma)\psi^\eps,\psi^\eps\right)_{L^2(M)}  = 
\frac 1\eps \left( {\rm Op}_\eps ([ T\sigma,\pi(-\mathbb L)])\psi^\eps,\psi^\eps\right)_{L^2(M)}\\
&\qquad= \frac 1\eps \left( [ {\rm Op}_\eps(T\sigma),-\eps^2 \mathbb L_M^{\mathbb U}] \psi^\eps,\psi^\eps\right)_{L^2(M)}
-2 \left( {\rm Op}_\eps((\pi(V)\cdot V ) \circ T \sigma)\psi^\eps,\psi^\eps\right)_{L^2(M)}+O(\eps).
\end{align*}
{
By~\eqref{zorro19}, the first term of the right-hand side is $0$ and we have }
\begin{align*}
&\frac 1\eps \left( {\rm Op}_\eps(\pi(V)\cdot V \sigma)\psi^\eps,\psi^\eps\right)_{L^2(M)}  = 
 - 2\left( {\rm Op}_\eps((\pi(V)\cdot V)\circ T \sigma)\psi^\eps,\psi^\eps\right)_{L^2(M)} +O(\eps).
\end{align*}
{Plugging this expression of $\left( {\rm Op}_\eps(\pi(V)\cdot V \sigma)\psi^\eps,\psi^\eps\right)_{L^2(M)}$ in~\eqref{A2} and using one more time~\eqref{zorro19}, we finally get}
\begin{align*}
    O(\eps)& =\frac 2\eps  \left( {\rm Op}_\eps(\pi(V)\cdot V\sigma)\psi^\eps,\psi^\eps\right)_{L^2(M)}+ \left({\rm Op}_\eps(\mathbb L\sigma)\psi^\eps,\psi^\eps\right)_{L^2(M)}\\
    &= \left({\rm Op}_\eps (-4 \pi(V)\cdot V\circ T \sigma +\mathbb L\sigma)\psi^\eps,\psi^\eps\right)_{L^2(M)}.
\end{align*}
We now { pass to the limit $\eps\rightarrow 0$ and transform the latter equation according to the equality 
\[
-4 \pi(V)\cdot V\circ T \sigma +\mathbb L\sigma=  i \sum_{\zeta\in{\rm Sp} \widehat {\mathbb L}}\nabla_\lambda  \zeta\, \sigma\, \widehat{\mathbb P}_\zeta.
\]
induced by  Corollary~\ref{cor:[P,Q]} and the fact that $\sigma\in\mathcal B_0$. We are left with 
\[
\sum_{\zeta\in{\rm Sp} \widehat {\mathbb L}}
\int_{M\times \Omega_0} {\rm Tr} \left( \nabla_\lambda  \zeta\,  \sigma  (x,\pi^{\lambda}) \Gamma_\zeta (x,\pi^{\lambda}) \right)  d\gamma(x,\pi^{\lambda}) 
=0,
\]
and the relation holds for all $\sigma=\sigma_1\eta$ with $\sigma_1\in\mathcal B_0$ and $\eta\in\mathcal S(\mathfrak z^*)$ supported in the dense open set~$\Omega_0$. This
 concludes the proof.}

  \section{ Geometric invariance  }
  
  In this section, we address the geometric invariance of the objects that we have introduced above. 

\subsection{ Nilmanifolds as  filtered manifolds }

A stratified Lie group $G$ carries a natural filtration on its Lie algebra given by
 $$\mathfrak h_1 \subset \mathfrak h_2 \subset \cdots \subset \mathfrak h_k=\mathfrak g= T_eG,\qquad\mbox{with}\qquad
 \mathfrak h_j=\mathfrak g_1\oplus \cdots\oplus \mathfrak g_j.
 $$
 \def \SVD {{\Gamma_c\left(\mathcal{S}(\mathbb{G}M), |\Lambda|\mathcal{V}\right)}}  
  One can view the nilmanifold 
 $M$  as a \textit{filtered manifold} with associated filtration of subbundles
  \begin{equation}
    H_{\dot{x}}^1\subset H_{\dot{x}}^2\subset \cdots \subset H_{\dot{x}}^r=T_{\dot{x}}M, \quad x\in G, \quad [H^i, H^j]\subset H^{i+j}, \; 1\leq i+j\leq r,\label{filtration}
  \end{equation}
  given by $H_{\dot{x}}^i=d\pi_\Gamma\circ dL_x(\mathfrak{h}_i)$.
  Here $\pi_\Gamma: G\to M=\Gamma \backslash G$ is the quotient map and $L_x: G\to G$ is the left-translation. In fact, $G$ induces a left-invariant stratification by the subbundles $d\pi_\Gamma\circ d_{L_x} (\mathfrak{g}_i)$ of $TM$ in the obvious way, but such a stratification will not respect the Lie bracket of vector fields on $M$ unless one restricts to left-invariant vector fields. What's more, we will see that the semi-classical calculus only depends on the filtration, and not on the stratification or the metric. 
  
When $G$ is step 2, we have $\mathfrak{h}_1=\mathfrak{v}$ and $\mathfrak{h}_2=\mathfrak{g}$. In this case, the data of the filtration on $G$ is almost the same as a stratification except that one forgets the second {stratum} $\mathfrak{g}_2=\mathfrak{z}$. On $M$, the filtration is given by a single step 2 bracket generating subbundle $H^1\subset TM$ without a preferred {complement}.

\subsection{Filtration preserving maps}

 Let $U$ be an open subset of $M$ and 
$\Phi: U\rightarrow M$ a smooth map on $M$.  
We introduce two definitions. 
\begin{definition}
\begin{enumerate}
\item The smooth map   
  $\Phi$ is said to preserve the filtration at ${\dot{x}}\in U$  when
$$ 
d_{\dot{x}}\Phi\left(H_{\dot{x}}^i\right)\subseteq H_{\Phi({\dot{x}})}^i, \quad i=1, \ldots, r.
$$
\item 
The map $\Phi$ is Pansu differentiable at the point  ${\dot{x}}$ when for any $z\in G$, 
\begin{equation}
    \label{pansuDerivative}
    \lim_{\eps\to 0}  \delta_{\eps^{-1}}\left( \Phi({\dot{x}})^{-1} \Phi({\dot{x}}\delta_\eps z)\right)= \lim_{\eps\to 0}  \delta_{\eps^{-1}}\left(  \Phi_{\dot{x}}(\delta_\eps z)\right) =: {\rm PD}_{\dot{x}} \Phi(z).
\end{equation}
\item The map $\Phi$ is uniformly Pansu differentiable on $U$ if it is Pansu differentiable at every point in $U$, and the limit \eqref{pansuDerivative} holds locally uniformly on $U\times G$. 
\end{enumerate}
\end{definition}
\begin{remark}
Taking $U\subset M$ to be a sufficiently small neighborhood of $\dot{x}\in M$, we may consider $U$ as a neighborhood of $x\in G$ and lift $\Phi$ to a smooth map $\Phi_G: U\subset G\to G$. Then the above definition is equivalent to saying $\Phi_G$ is  Pansu differentiable (resp. uniformly Pansu differentiable) at $x\in G$ (resp. on $U\times G$). 
\end{remark}
On a neighborhood $U\subset M$ sufficiently small to identify with a neighborhood in $G$, the notions of Pansu differentiability and filtration preservation are related via  the following result \cite{FFF}:

\begin{theorem} [\cite{FFF}]\label{pansuEquiv}
The map  $\Phi$ is uniformly Pansu differentiable on $U$ if and only if $\Phi$ preserves the filtration at every point $x\in U.$
\end{theorem}

This result relates a morally algebraic property, Pansu differentiability, to a geometric property of being filtration-preserving. Consequently, the diffeomorphisms  $\Phi$ we consider in the sequel are uniformly Pansu differentiable, and the transformation of pseudodifferential operators by the pull-back associated with $\Phi$ will involve the Pansu derivative of $\Phi$. This leads us to employ the osculating Lie group and Lie algebra bundles in the next section.

 \subsection{Schwartz Vertical Densities}
 For a filtered manifold $M$, the \textit{osculating Lie algebra bundle} $\mathfrak{G}M$ (and the \textit{osculatig Lie group bundle} $\mathbb{G}M:=\exp(\mathbb{G}M)$), defined in \cite{FFF}, play the role of the tangent bundle. When $M=\Gamma\backslash G$, with the filtration \ref{filtration},  the fibers of $\mathfrak G M$ and  $\mathbb GM$,  are all isomorphic to $\mathfrak g$ and $G$ respectively. In particular, we have canonical identifications
 \[
   \mathfrak GM\cong M\times \mathfrak g
 \;\;\mbox{and} \;\;
 \mathbb GM\cong M\times G.
 \]
 The  Haar measure on each fibers $\mathbb G_{\dot{x}} M$ is given by $d_{\dot{x}}z= dz$ and the dual  sets by $\widehat {\mathbb G}_{\dot{x}}M=\widehat G$.
 
To any semi-classical pseudodifferential operator on a compact nilmanifold  $M=\Gamma\backslash G$, its convolution kernel $\kappa$ may be viewed as an element of $C^\infty(M, \mathcal{S}(G))$.
However, this is not the right space for the general case of filtered manifolds. Indeed, Theorem \ref{changeofvar} together with Equation \eqref{transformationsymbols} below will imply that as a pseudodifferential operator transforms under diffeomorphisms on $M$ preserving the filtration, its associated convolution kernel transforms like a density on the osculating group bundle. We show that it is natural to view the convolution kernels $\kappa$ as elements of the bundle of \textit{Schwartz vertical densities} on $\mathbb{G}M$, rather than functions in $C^\infty(M, \mathcal{S}(G))$. We briefly elaborate below.  
 
 Let $\mathcal{V}(\mathbb{G}M)$ be the vertical bundle of $\mathbb{G}M$, that is, the kernel of the map $\mathbb{G}M\to M$. Let $|\Lambda|\mathcal{V}(\mathbb{G}M)$ be the bundle of vertical densities, that is, the bundle over $\mathbb{G}M$ whose fibers are densities in the vertical spaces. 
Let $\mathcal{S}(\mathbb{G}M)=\coprod_{{\dot{x}}\in M} \mathcal{S}(\mathbb{G}_{\dot{x}}M)$ be the Fr\'echet vector bundle over $M$ whose fibers are Schwartz class functions. Furthermore, let $\mathcal{S}(\mathbb{G}M, |\Lambda|\mathcal{V})$ be the Fr\'echet bundle over $M$ whose fibers are Schwartz class densities on the vertical space. As in \cite{FFF}, denote by $\Svd$, the space of its smooth compactly supported sections, which we call the Schwartz vertical densities. After making a choice of Haar measure on $G$, this space is identified with $C^\infty(M, \mathcal{S}(G))$. 

Indeed, by left-invariance, we identify the fibers of $\mathcal{V}(\mathbb{G}M)$  with the Lie algebra $\mathfrak{g}$ and fibers of $|\Lambda|\mathcal{V}(\mathbb{G}M)$ with $|\Lambda|\mathfrak{g}$, the set of densities on the Lie algebra:
\begin{equation}
\label{trivialization}
  \mathcal{V}(\mathbb G M)\cong (M\times G) \times \mathfrak{g}, \quad \text{ whence } \quad|\Lambda|\mathcal{V}(\mathbb G M)\cong (M\times G) \times |\Lambda|\mathfrak{g}.
\end{equation}
The above trivializations give the identification
\begin{align*}
    \Svd\cong C^\infty (M, \mathcal{S}(G, |\Lambda|\mathfrak{g}))
\end{align*}
And a choice of Haar measure on $G$ gives $\mathcal{S}(G, |\Lambda|\mathfrak{g})\cong \mathcal{S}(G)$.
\medskip

For a choice of Haar measure $dz$ on $G$, which in turn gives a Haar system $\{d_{\dot{x}}z\}$ on $\mathbb{G}M$ through \eqref{trivialization}, the identifications of vertical Schwartz densities with functions is given explicitly by
$$
\kappa\in  \Svd: {\dot{x}}\mapsto \kappa_{\dot{x}}=\widetilde{\kappa}_{\dot{x}}d_{\dot{x}}z, \quad \widetilde{\kappa}_{\dot{x}}\in \mathcal S(\mathbb G_{\dot{x}}M).
$$
The  symbols $\sigma\in\mathcal A_0$ are defined as  the images of the elements  $\kappa\in \Svd$ by the fiberwise Fourier transform:
\begin{equation}
\label{symbol}
    {\dot{x}}\mapsto \sigma({\dot{x}},\pi) =\int_{z\in \mathbb{G}_{\dot{x}}M} \widetilde{\kappa}_{\dot{x}}(z) \pi(z)^* d_{\dot{x}} z ,\;\;\pi\in \widehat{\mathbb{G}}_{\dot{x}}M.
\end{equation}
Since our convolution kernels are densities, the integral \eqref{symbol} is independent of the choice of Haar measure.

\subsection{Semi-classical pseudodifferential calculus and  filtration diffeomorphisms} We keep the notations of the preceding section except we suppose $\Phi:U\subset M \to M$ is a diffeomorphism onto its image. Let $J_\Phi$ be the Jacobian of $\Phi$. We associate with $\Phi$ 

\begin{enumerate}
\item[(i)] 
  a unitary transformation $\PullTwo$ of $L^2(U)$ induced by $\Phi$ 
\[
\PullTwo(f) :=  J_\Phi^{1/2} f\circ \Phi,\;\; f\in L^2(U),
\]
\item[(ii)]  a map $\PullOne$ on the space of Schwartz vertical densities that extends to an isometry of $L^1(|\Lambda|\mathcal{V}(\mathbb{G}M))$
$$
(\PullOne \kappa)_x(z) := J_\Phi(x)  \, \kappa_{\Phi(x)} ( {\rm PD}_x\Phi(z)),\;\;\forall (x,z)\in U\times G.
$$
\end{enumerate}
We are interested in the properties of the operator $\PullTwo\circ  {\rm Op}_\eps (\sigma)\circ \PullTwo^{-1}$, in particular in the asymptotics in~$\eps$ of its semi-classical pseudodifferential symbol. The structure of the latter and the way it can be deduced from $\sigma$ will give information of the geometric nature of the objects. Indeed,  $\Phi$ induces several geometric transformations:
\begin{enumerate}
\item[(i)] 
  $\Phi$ induces a map on  representations
$$
    \widehat{\mathbb{G}}\Phi: 
    \left\{
    \begin{array}{rcl}
    U\times\widehat{G}&\longrightarrow& \Phi(U)\times\widehat{G}\\
    (x, \pi)&\longmapsto& \left(\Phi(x), \pi\circ ({\rm PD}_x\Phi)^{-1}\right)
\end{array}\right.
$$
\item[(ii)]  The \textit{generalized canonical transformation} $\widehat{\mathbb{G}}\Phi$  induces a pull-back on symbols
$$
(\widehat{\mathbb{G}}\Phi)^*\sigma(x, \pi):=\sigma(\widehat{\mathbb{G}}\Phi(x, \pi)).
$$
\end{enumerate}
The maps $\widehat{\mathbb{G}}\Phi$ and $\mathcal{I}_\Phi$ are intertwined by the group Fourier Transform: If $\sigma (x,\pi) = \widehat {\kappa_x}(\pi)$ for all $x\in U\subset M$ and $\pi\in \widehat{\mathbb{G}}_xM$, then for any filtration preserving differmorphism $\Phi: U\to M$
\begin{equation}
\label{transformationsymbols}
   (\widehat{\mathbb{G}}\Phi)^*\sigma(x, \pi) = \widehat{\PullOne\kappa_x}(\pi), \quad x\in U\subset M, \; \pi\in\widehat G=\widehat{\mathbb G}_xM. 
\end{equation}
 These two maps are involved in the description of the first term of the expansion of the semi-classical symbol of the operator $\PullTwo\circ  {\rm Op}_\eps (\sigma)\circ \PullTwo^{-1}$:

\begin{theorem}[\cite{FFF}]\label{changeofvar}
Assume that $\Phi$ is filtration preserving on~$U$. 
Then  in~$\mathcal L(L^2(U))$,
\vskip -0,2cm
$$
\PullTwo\circ  {\rm Op}_\eps (\sigma)\circ \PullTwo^{-1} ={\rm Op}_\eps\left((\widehat{\mathbb{G}}\Phi)^* \sigma \right) +O(\eps).
$$  
\end{theorem} 
\begin{remark}
Theorem \ref{changeofvar} establishes the geometric invariance of the semi-classical calculus by filtration preserving differmorphisms $\Phi$. In particular, $\Phi$ does not need to preserve the action of $G$ on $M$, or even preserve the gradation. 
\end{remark}
The results of this section suggest that the semi-classical symbols  we defined in Section \ref{subsec_pseudodiff} ought to be the natural  generalization of symbols for arbitrary filtered manifolds. So defined, the semi-classical symbols are invariant under generalized canonical transformations of $\widehat{\mathbb{G}}M$  associated to differmorphisms preserving the filtration on $M$.

\section{Technical tools}\label{sec:tech}

This section is devoted to several technical results used in the proof of Theorem~\ref{theo:main2}. 

\subsection{Some $C^*$-algebras and their properties}
\label{subsec_B}

\subsubsection{The von Neumann algebra $L^\infty(M\times \widehat G)$}
 A measurable symbol $\sigma=\{\sigma(x,\pi): (x,\pi)\in M\times \widehat G\}$ is said to be bounded when  there exists a constant $C>0$ such that for $dxd\mu(\pi)$-almost all $(x,\pi)\in M\times \widehat G$, we have
 $ \|\sigma(x,\pi)\|_{\mathcal H_{\pi}} \leq C $.
 We denote by $\|\sigma\|_{L^{\infty}(M\times \widehat G)}$ the smallest of such constant $C>0$ and by $L^{\infty }(M\times \widehat G)$ 
 the space of bounded measurable symbols.
 We check readily that $\|\cdot\|_{L^{\infty}(M\times \widehat G)}$ is a norm on    $L^{\infty }(M\times \widehat G)$ which is  a $C^{*}$-algebra.
 We will later use the fact that it is a von Neumann algebra. 
 
 \subsubsection{The $C^*$-algebra $\mathcal A$ and its topological dual}
 Clearly, $\mathcal A_{0}$ is a subspace of $L^{\infty }(M\times \widehat G)$.
 Its closure  denoted by $\mathcal A$   for the norm 
$\|\cdot\|_{L^{\infty}(M\times \widehat G)}$  is a sub-$C^{*}$-algebra of $L^{\infty }(M\times \widehat G)$.
Its topological dual $\mathcal A^{*}$ is isomorphic to  the Banach space of operator-valued measures $\mathcal M_{ov}(M\times \widehat G)$ via 
$$
\mathcal M_{ov}(M\times \widehat G)\ni 
\Gamma d\gamma \mapsto \ell_{\Gamma d\gamma}\in \mathcal A^{*}, 
\qquad
\ell_{\Gamma d\gamma}(\sigma):=\int_{M\times \widehat G}
{\rm Tr} \left(\sigma(x,\pi) \Gamma (x,\pi) \right) d\gamma(x,\pi).
$$
Moreover, the isomorphism is isometric:
$$
\|\ell_{\Gamma d \gamma}\|_{\mathcal A^{*}} =\|\Gamma d\gamma\|_{\mathcal M_{ov}(M\times \widehat G)}, 
\quad\mbox{where}\quad 
\|\Gamma d\gamma\|_{\mathcal M_{ov}}:=
 \int_{M\times \widehat G}
{\rm Tr} \left|\Gamma (x,\pi) \right| d\gamma(x,\pi),
$$
and the positive linear functionals on $\mathcal A$ are the $\ell=\ell_{\Gamma d\gamma}$'s with  $\Gamma d\gamma \geq 0$. 

\subsubsection{The $C^*$-algebra $\mathcal B$ and its topological dual}
Let $\mathcal B_{0}$ be the  subspace of $\mathcal A_{0}$ of symbols commuting with $\widehat {\mathbb L}$.
Clearly 
$\mathcal B_{0}$
 contains all the symbols of the form $a(x) \psi(\widehat {\mathbb L})$, $a\in \mathcal C^{\infty}(M)$, $\psi\in \mathcal S(\R)$, 
by Hulanicki's theorem (see~\cite{hul}):

\begin{theorem}[Hulanicki]
\label{thm_hula}
	The convolution kernel of a spectral multiplier $\psi(\mathbb L_G)$ of $\mathbb L_G$ for a Schwartz function 
 $\psi\in \mathcal S(\R)$ is Schwartz on $G$.
\end{theorem}

We denote by $\mathcal B$ the closure of $\mathcal B_{0}$   for the norm 
$\|\cdot\|_{L^{\infty}(M\times \widehat G)}$.
Property  \eqref{item_sp} of $\widehat {\mathbb L}$ recalled in Section \ref{subsubsec_FG}
 implies that 
 $\mathcal B$ is the  subspace of $\mathcal A$ of symbols commuting with every $\widehat{\bP}_{\zeta}$, $\zeta>0$.
 We check readily  that $\mathcal B$ is a sub-$C^{*}$-algebra of $\mathcal A$ and that  $\mathcal B_{0}=\mathcal A_{0}\cap \mathcal B$.
 The next statement identifies the topological dual of $\mathcal B$:

 \begin{proposition}
 \label{prop_Bdual}
  Via $\Gamma d\gamma \mapsto \ell_{\Gamma d\gamma}|_{\mathcal B}$,
 the topological dual $\mathcal B^{*}$ of $\mathcal B$ is isomorphic 
 with the closed subspace $\mathcal M_{ov}(M\times \widehat G)^{(\widehat{\mathbb L})}$ of operator valued measures $\Gamma d\gamma \in \mathcal M_{ov}(M\times \widehat G)$ such that the operator~$\Gamma$ commutes with $\widehat \bP_\zeta$
 for all $\zeta>0$,
 in the sense that
  $$
  \forall \zeta> 0 \qquad 
 \pi( \bP_{\zeta}) \Gamma (x,\pi) =  \Gamma (x,\pi)\pi( \bP_{\zeta}) 
 \quad \mbox{for} \ \gamma-\mbox{almost all}\ (x,\pi) \in M\times \widehat G.
 $$
 \end{proposition}

 \begin{proof}
 \textbf{Step 0.}
We observe that if two pairs $(\Gamma,\gamma)$ and $(\Gamma_{1},\gamma_{1})$ are equivalent and one of them satisfies the commutative condition with $\widehat \bP_\zeta$ for all $\zeta>0$,
then so does the other. Hence, $\mathcal M_{ov}(M\times \widehat G)^{(\widehat{\mathbb L})}$ is a well defined subset of $\mathcal M_{ov}(M\times \widehat G)$.
 One checks that it is a closed subspace of $\mathcal M_{ov}(M\times \widehat G)$. 
 \smallskip 
 
\noindent  \textbf{Step 1.} Let $\ell\in \mathcal B^{*}$. By the Hahn-Banach theorem, 
 this functional extends to  $\tilde \ell \in \mathcal A^{*}$, i.e. $\tilde\ell_{|\mathcal B} = \ell$.
 Denote by  $\Gamma d \gamma \in \mathcal M_{ov}(M\times \widehat G)$ 
the corresponding operator-valued measure: $\tilde \ell = \ell_{\Gamma d\gamma}$. 
Now set 
$$
\Gamma_{1} (x,\pi)
:= \sum_{\zeta\in {\rm sp}(\pi(\mathbb L) )}\pi(\bP_{\zeta}) \Gamma (x,\pi) \pi(\bP_{\zeta})   .
$$
The operator-valued measure $\Gamma_{1}d\gamma$ is a well defined  element of  $\mathcal M_{ov}(M\times \widehat G)$ satisfying the condition of commutativity with $\widehat \bL$   so  $\Gamma_{1}d\gamma \in \mathcal M_{ov}(M\times \widehat G)^{(\widehat{\mathbb L})}$.
 Let us show that it  coincides with $\ell_{\Gamma d\gamma}$ on $\mathcal B$. 
 Let $\sigma\in \mathcal B$. 
 Since $\sum_{\zeta\in {\rm sp}(\pi(\mathbb L) )}\pi(\bP_{\zeta}) $ is the identity operator on $\mathcal H_{\pi}$, we have
 \begin{align*}
  \ell_{\Gamma d\gamma}(\sigma)
  &=\int_{M\times \widehat G} \sum_{\zeta\in {\rm sp}(\pi(\mathbb L) )}
{\rm Tr} \left(\sigma(x,\pi) \Gamma (x,\pi) \pi(\bP_{\zeta})\right) d\gamma(x,\pi)\\
  &=\int_{M\times \widehat G} \sum_{\zeta\in {\rm sp}(\pi(\mathbb L)- }
{\rm Tr} \left(\sigma(x,\pi) \pi(\bP_{\zeta}) \Gamma (x,\pi) \pi(\bP_{\zeta})\right) 
d\gamma(x,\pi),) \end{align*}
 since $\pi(\bP_{\zeta}) = \pi(\bP_{\zeta})^{2}$ commutes with $\sigma(x,\pi)$. We recognise $\ell_{\Gamma_{1}d\gamma}(\sigma)$ on the right-hand side. 
 We have obtained that any $\ell\in \mathcal B^{*}$ may be written as the restriction to $\mathcal B$ of $\ell_{\Gamma_{1}d\gamma}$, for some 
 $\Gamma_{1}d\gamma \in \mathcal M_{ov}(M\times \widehat G)^{(\widehat{\mathbb L})}$. 
 
 \smallskip

 In other words, we have  proved  that  $\Gamma d\gamma \mapsto \ell_{\Gamma d\gamma}|_{\mathcal B}$ maps $ \mathcal M_{ov}(M\times \widehat G)^{(\widehat{\mathbb L})}$ onto~$\mathcal B^*$.
 This map is continuous and linear. It remains to show that it is injective. 

 \smallskip
 
\noindent  \textbf{Side step.} Let us open a parenthesis.
 The von Neumann algebra
 $L^{\infty }(M\times \widehat G)$ is a $C^{*}$ algebra containing $\mathcal A$ and  we denote by $vN\mathcal A$  the von Neumann algebra generated by $\mathcal A$.
This means that  $vN\mathcal A$ is the closure of $\mathcal A$ for the strong operator topology  in $L^{\infty }(M\times \widehat G)$.
We are going to use this 
 von Neumann algebra by considering the natural unique extension of 
 $\ell=\ell_{\Gamma d\gamma} \in \mathcal A^{*}$ 
 to a continuous linear functional on the von Neumann algebra $vN\mathcal A$ of~$\mathcal A$. 

Since  $\mathcal B \subset \mathcal A$, 
 we also have $vN\mathcal B \subset vN\mathcal A$ where
$vN\mathcal B$  denotes
the von Neumann algebra generated by $\mathcal B$.  
Moreover, $vN\mathcal B$ is the subspace of the symbols $\sigma\in vN\mathcal A$ commuting with  $\widehat \bP_{\zeta}$ $dxd\mu(\pi)$-almost everywhere for every , $\zeta>0$. 

Finally, we observe that for $\zeta>0$ and $\sigma\in \mathcal A$, the symbol $\pi(\widehat \bP_{\zeta} )\sigma\pi( \widehat \bP_{\zeta} )$ is in  $vN\mathcal A$.
Indeed,  using
Hulanicki's theorem (Theorem \ref{thm_hula})
together with $\mathcal S(G)*\mathcal S(G)\subset \mathcal S(G)$, we obtain that if $\sigma\in \mathcal A_{0}$ then for any $\psi_{1},\psi_{2}\in \mathcal S(\R)$, the symbol $\psi_{1}(\widehat {\mathbb L})\sigma \psi_{2}(\widehat {\mathbb  L})$ is in $\mathcal A_{0}$.
Taking limits for suitable sequences of $\sigma,\psi_{1},\psi_{2}$ implies that  the symbol $\pi(\widehat \bP_{\zeta} )\sigma\pi( \widehat \bP_{\zeta} )$ is in  $vN\mathcal A$ for any $\sigma\in \mathcal A$.

 \smallskip

\noindent \textbf{Step 2.} Let us now consider $\Gamma d\gamma  \in \mathcal M_{ov}(M\times \widehat G)^{(\widehat{\mathbb L})}$  such that $\ell:=\ell_{\Gamma d \gamma }$ vanishes  on $\mathcal B$. We want to show $\ell=0$. We extend $\ell$ to a functional $L$ on $vN\mathcal A$. This functional vanishes on $vN\mathcal B$. We set
 $$
 L_{\zeta}(\sigma):= \int_{M\times \widehat G}{\rm Tr} \left(\sigma(x,\pi) \pi(\bP_{\zeta}) \Gamma (x,\pi) \right) d\gamma(x,\pi) , \qquad \zeta>0.
 $$
 We check readily that $\zeta\mapsto L_{\zeta}(\sigma)$ defines a complex measure on $[0,\infty)$ with total mass that is smaller or equal to $ \|\sigma\|_{L^{\infty}(M\times \widehat G)}\|\Gamma d\gamma\|_{\mathcal M_{ov}}$.
 Moreover, $ \ell(\sigma) = \int_{0}^{+\infty} L_{\zeta}(\sigma)$
 since $\sum_{\zeta\in {\rm sp}(\pi(\mathbb L) }\pi(\bP_{\zeta}) $ is the identity operator on $\mathcal H_{\pi}$.
 
Using $\bP_{\zeta}^2= \bP_{\zeta}$ and the commutation of $\Gamma$ with $\pi(\bP_{\zeta})$ $d\gamma$-a.e., together with trace property, we obtain 
  $$
 L_{\zeta}(\sigma)= \int_{M\times \widehat G}{\rm Tr} \left( \pi(\bP_{\zeta}) \sigma(x,\pi) \pi(\bP_{\zeta}) \Gamma (x,\pi) 
 \right) d\gamma(x,\pi)= L_\zeta(\pi(\bP_{\zeta}) \sigma(x,\pi) \pi(\bP_{\zeta}))
 $$
 with $\pi(\bP_{\zeta}) \sigma(x,\pi) \pi(\bP_{\zeta})\in vN\mathcal B$. Arguing as above (in the side step), we deduce
 $L_{\zeta}=0$, whence $L=0$ and $\ell=0$. 
 This implies the injectivity of 
  $\Gamma d\gamma \mapsto \ell_{\Gamma d\gamma}|_{\mathcal B}$  on $\mathcal M_{ov}(M\times \widehat G)^{(\widehat{\mathbb L})}$.  
 \end{proof}
 
 The proof above has an important consequence regarding the restriction of symbols to $M\times \widehat G_{1}$, a notion we now explain. 
 
 \subsubsection{Restriction of symbols to $M\times \widehat G_{1}$}
 The restriction $\sigma |_{M\times \widehat G_{1}}$ of $\sigma \in \mathcal A$ with kernel $ \kappa_{x}(y)$, 
 to $M\times \widehat G_{1}$ is given by
 $$
 \sigma |_{M\times \widehat G_{1}}(x,\omega)=
 \sigma(x,\pi^{\omega}) = \mathcal F_{\mathfrak v} \int_{\mathfrak z}\kappa_{x}(\cdot ,z) dz(\omega),
 \quad (x,\omega)\in M\times \mathfrak v^{*},
 $$
 having identified $\widehat G_{1}$ with $\mathfrak v^{*}$.
 Moreover, 
 we can therefore identify 
 $$
 \mathcal  A|_{M\times \widehat G_{1}}:=\{\sigma |_{M\times \widehat G_{1}}, \sigma \in \mathcal A\}
 $$
with a sub-space of $\mathcal C_{0}(M\times \mathfrak v^{*})$. 
In fact, we can  show 
\begin{lemma}
\label{lem_rest=C0}
	We have $
\mathcal C_{0}(M\times \mathfrak v^{*}) = \mathcal  A|_{M\times \widehat G_{1}}.$
\end{lemma}

\begin{proof}
Any element of $\mathcal C_{0}(M\times \mathfrak v^{*})$ may be viewed as  a limit for the supremum norm on $M\times \mathfrak v^{*}$ of  $\mathcal F_{\mathfrak v}\kappa^{(j)}_{x}(\omega)$ for a sequence of  kernels $\kappa^{(j)} \in \mathcal C^{\infty}(M, \mathcal S(\mathfrak v))$. 
We then 
consider the sequence of symbols $\sigma_{j} (x,\pi) = \pi (\kappa^{(j)}_{x}\eta) $
with $\eta\in \mathcal S(\mathfrak z)$ satisfying $\mathcal F_{\mathfrak z}\eta(0)=\int_{\mathfrak z}\eta (Z)dz=1$. We check readily that 
$\sigma_{j} |_{M\times \widehat G_{1}}(x,\omega)=\mathcal F_{\mathfrak v}\kappa^{(j)}_{x}(\omega)$. 	
\end{proof}

For  a subspace $S$ of $\mathcal A$,  we denote by 
 $$
 S|_{M\times \widehat G_{1}}:=\{\sigma |_{M\times \widehat G_{1}}, \sigma \in S\}
 $$
 the resulting subspace in $\mathcal  A|_{M\times \widehat G_{1}}$.
 The proof of Lemma \ref{lem_rest=C0}
 shows that if $\bar S$ denotes the closure of $S$ in the $C^{*}$-algebra $\mathcal A$, then $\bar S|_{M\times \widehat G_{1}}$ is the closure of 
 $S|_{M\times \widehat G_{1}}$ in $\mathcal C_{0}(M\times \mathfrak v^{*})$, 
 that is, given by the supremum norm on $M\times \mathfrak v^{*}$.
 Hence $\bar S|_{M\times \widehat G_{1}} \subset \mathcal  A|_{M\times \widehat G_{1}}$.

 We will need the following property regarding the restriction of the symbols
in $\mathcal A_{0}$ and $\mathcal B_{0}$
 to $M\times \widehat G_{1}$;
 its proof relies on the proof of Proposition \ref{prop_Bdual}:

 \begin{corollary}
 \label{cor_prop_Bdual}
 The following commutative $C^{*}$ algebras coincide:
$$
\bar {\mathcal B}_{0} |_{M\times \widehat G_{1}} = \mathcal  B|_{M\times \widehat G_{1}} =
\bar {\mathcal  A}_{0}|_{M\times \widehat G_{1}}  = \mathcal  A|_{M\times \widehat G_{1}}=\mathcal C_{0}(M\times \mathfrak v^{*}).
$$
 \end{corollary}

\begin{proof}
Clearly, $\bar {\mathcal B}_{0} |_{M\times \widehat G_{1}}
=  \mathcal  B|_{M\times \widehat G_{1}} \subset 
\bar {\mathcal  A}_{0}|_{M\times \widehat G_{1}} = \mathcal  A|_{M\times \widehat G_{1}}=\mathcal C_{0}(M\times \mathfrak v^{*})$.
It remains to show  the converse inequality. 

Let $\ell $ be a continuous linear functional on $\mathcal C_{0}(M\times \mathfrak v^{*})$.
This is given by integration against a complex Radon measure $\gamma_{1}$.
Consider the operator-valued measure $\Gamma d\gamma\in \mathcal M_{ov}(M\times\widehat G)$ defined by
$1_{M\times\widehat G_{\infty}}\Gamma d\gamma=0$ and 
$1_{M\times\widehat G_1}\Gamma d\gamma=\gamma_{1}$, 
that is, 
$$
\ell_{\Gamma d\gamma}(\sigma) = \int_{M\times \mathfrak v^{*}}
\sigma |_{M\times \widehat G_{1}}
 (x,\pi^\omega) \ d\gamma_{1}(\omega),
 \quad \sigma\in \mathcal A.
$$
We observe that  $\Gamma$ commutes with $\widehat \bP_{\zeta}$, $\zeta>0$.
Hence, if  $\ell=0$ on $\mathcal  B|_{M\times \widehat G_{1}}$ then   $\ell_{\Gamma d\gamma}\equiv 0$ on  $\mathcal  B$
and therefore also on $\mathcal A$ 
 by Proposition \ref{prop_Bdual}, or rather Step 2 of its proof;
 this implies $\Gamma d\gamma =0$ thus $\gamma_1=0$ and $\ell=0$.
 By the Hahn-Banach theorem, this shows that $\mathcal  B|_{M\times \widehat G_{1}} =\mathcal C_{0}(M\times \mathfrak v^{*}) $.
\end{proof}

\subsubsection{Some elements of $vN\mathcal A$ and $vN\mathcal B$}
We will need the following properties:
\begin{lemma}
\label{lem_vNstuff}
If $\sigma \in \mathcal A_0$
then ${\bf 1}_{M\times \widehat G_1} \sigma$
and  ${\bf 1}_{M\times \widehat G_\infty} \sigma$
are in $vN\mathcal A$.
Similarly, if $\sigma \in \mathcal B_0$
then ${\bf 1}_{M\times \widehat G_1} \sigma$
and  ${\bf 1}_{M\times \widehat G_\infty} \sigma$
are in $vN\mathcal B$.
\end{lemma}

\begin{proof}
We consider $\sigma \eta $ as in Lemma \ref{lem_sigmaeta} with a sequence of functions $\eta\in \mathcal S(\mathfrak z^*)$ satisfying $\eta(0)=1$ and with support shrinking to $\{0\}$.
We check readily that if $\sigma\in \mathcal A_0$ then the limit of these $\sigma \eta $ for the strong operator topology will be 
${\bf 1}_{M\times \widehat G_1} \sigma$ which is therefore  in $vN\mathcal A$.
It will also be the case for 
  ${\bf 1}_{M\times \widehat G_\infty} \sigma = \sigma - {\bf 1}_{M\times \widehat G_1}\sigma$.
The case of $\mathcal B_0$ follows.
\end{proof}

\subsection{The lowering and raising operators associated with $H(\lambda)$. }
\label{subsec_opW}

\subsubsection{Preliminaries}
\label{subsubsec_preltensor}
Before proving several useful identities, we introduce some notations. 
If $\pi_1,\pi_2$ are two representations of $\mathfrak g$, and $A:\mathfrak v\to \mathfrak v$ is a linear morphism, then we set
$$
(A\pi_1(V))\cdot \pi_2(V) = \sum_{j,k} A_{j,k} \pi_1(V_k) \otimes  \pi_2(V_j) \in \mathcal H_{\pi_1} \otimes \mathcal H_{\pi_2} 
$$
where $(A_{j,k})$ is the matrix representing $A$ in the orthonormal basis $(V_j)$.
We can check that this is independent of the orthonormal basis $(V_j)$.
If the context is clear, we may allow ourselves to omit the notation for the tensor product $\otimes$ and may swap the order in the tensor product.

With $A={\rm id}_{\mathfrak v}$, $\pi_1$ being the regular representation of $\mathfrak g$ on $L^2(M)$ and $\pi_2=\pi\in \widehat G$, this yields the super-operator  $V \cdot \pi(V)$ acting on $\mathcal A_0$. 
If we restrict this to $M\times \widehat G_1$, i.e.  $\pi_2=  \pi^\omega \in \widehat G_1$, this defines $\omega \cdot V$ acting on $\mathcal C^\infty (M,\mathcal S(\mathfrak v^*) )\sim \mathcal A_0|_{M\times \widehat G_1} $.

  \subsubsection{Technical computations}
  \label{subsubsec_comput}
Here, we  assume that $k=0$ and consider $\lambda\in \Omega_0$.
Following Appendix~B in~\cite{FF3}, instead of the basis $P^\lambda_j,Q^\lambda_j$, $1\leq j\leq d$, 
we will use the fields
\begin{equation}\label{def:Rj}
W_j^\lambda := \frac 12 (P_j^\lambda-iQ_j^\lambda)\qquad\mbox{and} \qquad
\overline W_j^\lambda :=\frac 12 (P_j^\lambda +iQ_j^\lambda).  
\end{equation}
{Direct computations show
using equation~\eqref{eq:pi(P)},
$\pi(P^\lambda_j)=\sqrt{\eta_j(\lambda)} \partial_{\xi_j}$ and
$\pi(Q^\lambda_j)=i\sqrt{\eta_j(\lambda)} \xi_j$, so we obtain}
\[
\pi^\lambda(W_j^\lambda)= \frac{\sqrt{\eta_j(\lambda)}}2 (\partial_{\xi_j} +\xi_j) \qquad\mbox{and}\qquad \pi^\lambda (\overline W_j^\lambda)= \frac{\sqrt{\eta_j(\lambda)}}2 (\partial_{\xi_j} -\xi_j).
\] 
In particular, 
these new fields coincide up to normalisation with   the lowering and raising operators of the harmonic oscillators $(-\partial_{\xi_j}^2+\xi_j^2)$. 
Consequently,  the family of Hermite functions $(h_\alpha)_{\alpha\in\N^d}$ given by 
{
$$
 h_\alpha (\xi_1,\ldots, \xi_d) = 
 h_{\alpha_1}(\xi_1) \ldots h_{\alpha_d}(\xi_d), 
 \quad\mbox{where}\quad
 h_n(\xi) = \frac{(-1)^n}{\sqrt{2^n n! \sqrt \pi}}
 e^{\frac {\xi^2}2}
 \frac{d }{d\xi} ( e^{-\xi^2}), \quad n\in \N, 
 $$
 }
 is an orthonormal basis of~$L^2(\R^d) $ that satisfies:
 \begin{equation}
 \label{eq_Wh}
 	\pi^\lambda (W_j^\lambda) h_\alpha
= \sqrt{\frac{\eta_j(\lambda)}2 }
\sqrt{\alpha_j} h_{\alpha-{\bf 1}_j}
\qquad
\pi^\lambda(\overline W_j^\lambda) h_\alpha
= -\sqrt{\frac{\eta_j(\lambda)}2}
\sqrt{\alpha_j+1} h_{\alpha+{\bf 1}_j}.
 \end{equation}
Here, ${\bf 1}_j$  denotes the multi-index with $j$-th coordinate $1$ and $0$ elsewhere.
We also have extended the notation $h_\alpha$ to $\alpha\in \mathbb Z^d$ with $h_\alpha=0$ if $\alpha \notin \N^d$. 
{We then deduce easily
\begin{equation}
	\label{eq:zorro2}
    \left[\pi^\lambda (W_j^\lambda),  \pi^\lambda (-\mathbb L)\right]
    =2\eta_j(\lambda) \pi^\lambda (W_j^\lambda)
    \quad\mbox{and}\quad 
      \left[\pi^\lambda (\overline W_j^\lambda), \pi^\lambda (-\mathbb L)\right]
      =-2\eta_j(\lambda)\pi^\lambda (\overline W_j^\lambda).
\end{equation}
and  that both the operators
$\widehat {\mathbb P} _\zeta \pi(W^\lambda_j )\widehat {\mathbb P}_\zeta $ and $\widehat {\mathbb P} _\zeta \pi(\overline W^\lambda_j )\widehat {\mathbb P}_\zeta $ are zero. Consequently, we also have}
\begin{equation} \label{non_com_PQ}
\widehat {\mathbb P} _\zeta \pi(P^\lambda_j )\widehat {\mathbb P}_\zeta =0 \;\;\mbox{and}\;\;
\widehat {\mathbb P}_\zeta \pi(Q^\lambda_j )\widehat {\mathbb P}_\zeta =0,\;\forall \lambda\in \mathfrak z^*\setminus\{0\},\;\forall j\in\{1,\cdots, d\},\;
\forall \zeta\in\R,
\end{equation}

\medskip 

 Following the ideas and notation from Section \ref{subsubsec_preltensor}, 
we define the operator
$$
T= \frac i 2  (B(\lambda)^{-1}V \cdot \pi^\lambda(V)), 
\quad 
\lambda \in \Omega_0,
$$
acting on the space of symbols in $\mathcal A_0$ restricted to $M\times \Omega_0$. This may also be viewed as acting on 
the space of symbols in $\mathcal A_0$
which are supported in $M\times \Omega_0$.
The properties above imply:

\begin{lemma}\label{lem:computation}
\begin{enumerate}
\item 
For any $\sigma\in \mathcal A_0$,
we have on $M\times \Omega_0$:
$$
[T \sigma , \pi(-\mathbb L)] 
=\pi(V)\cdot V \sigma
$$
\item
For any $\lambda\in \Omega_0 $ and  $\zeta>0$,  
using the shorthand $\pi(\bP_\zeta) $ for ${\rm id}_{L^2(G)} \otimes \pi(\bP_\zeta)$, 
we have 
$$
\pi^\lambda(\bP_\zeta)
 \Big(V \cdot \pi^\lambda(V))\circ T \Big)
\pi^\lambda(\bP_\zeta) =
\frac{1} 4 \mathbb L - \frac i4 \sum_{j=1}^d (2\alpha_j+1) [P_j^\lambda,Q_j^\lambda
]
\pi^\lambda(\bP_\zeta).
$$
\end{enumerate}
\end{lemma}

\begin{proof}
 Since 
$P_j^\lambda= \overline W_j^\lambda + W_j^\lambda$, 
and $Q_j^\lambda= \frac 1{i}(\overline W_j^\lambda - W_j^\lambda),$
we deduce for $\pi=\pi^\lambda$, $\lambda\in\Omega_0$,
\begin{align*}
	    V\cdot \pi(V)  
	    &=
     2 \sum_{j=1}^d \left(W^\lambda_j   \pi(\overline W^\lambda_j) + 
    \overline W^\lambda_j  
    \pi(W^\lambda_j) \right).
    \end{align*}
 As $B(\lambda) Q_j^\lambda=\eta_j(\lambda) P^\lambda_j$ and $B(\lambda)P^\lambda_j=-\eta_j(\lambda) Q^\lambda_j$,  we obtain
    \begin{align*}
    (B(\lambda)^{-1}V) \cdot \pi(V)
    &= \sum_{j}^d \frac 1{\eta_j} \left(-P_j^\lambda\pi(Q_j^\lambda)+ Q_j^\lambda \pi(P_j^\lambda)  \right)=\frac 2 i
\sum_{j=1}^d \frac 1{\eta_j} \left(\overline W_j^\lambda \pi(W_j^\lambda)-W_j^\lambda\pi(\overline W_j^\lambda) \right).
\end{align*}
By \eqref{eq:zorro2}, we check readily 
Part (1).

\smallskip

For Part (2), 
we may assume that 
$\pi^\lambda(\bP_\zeta)\neq0$, that is, $\zeta$ is in the spectrum of the harmonic oscillator $\pi^\lambda (\mathbb L)$, or in other words $\zeta=\sum_j \eta_j(\lambda)(2\alpha_j+1)$ for some $\alpha\in \N^d$.
For any such index $\alpha$
and for an arbitrary vector $w_1\in \mathcal S(G)$,
by the computations above and \eqref{eq_Wh}, we see with $\pi=\pi^\lambda$:
\begin{align*}
& \pi(\bP_\zeta)  \big (V \cdot \pi(V)\big) \circ
\big(B(\lambda)^{-1}V \cdot \pi(V)\big) w_1\otimes h_\alpha
\\
&\qquad =
\frac 4i \sum_{j_1,j_2} \eta_{j_2}^{-1} \pi(\bP_\zeta) 
\big(W_{j_1}^\lambda \pi( \overline W_{j_1}^\lambda ) + 
\overline W_{j_1}^\lambda \pi(  W_{j_1}^\lambda )\big)
\big(\overline W_{j_2}^\lambda \pi(  W_{j_2}^\lambda ) - 
W_{j_2}^\lambda \pi(\overline   W_{j_2}^\lambda )\big)
w_1\otimes h_\alpha
\\
&\qquad =
 \frac 2i   \sum_{j}  \left(
\overline W_j^\lambda W_j^\lambda
(\alpha_j+1) -
W_j^\lambda\overline W_j^\lambda 
\alpha_j \right) 
w_1\otimes h_\alpha.
\end{align*}
We can simplify each term in the sum above with:
$$
\overline W_j^\lambda W_j^\lambda
(\alpha_j+1) -
W_j^\lambda\overline W_j^\lambda 
\alpha_j
=	
 \frac 14 ((P_j^\lambda)^2 + (Q_j^\lambda)^2 )
- \frac i 4 (2\alpha_j+1) [P_j^\lambda,Q_j^\lambda].
$$
Part (2) follows  
\end{proof}

{
We recall that the maps $\lambda\mapsto \eta_j(\lambda)$, $j=1,\ldots, d$,  are smooth in $\Lambda_0$. Moreover,  if $\lambda_0\in \Lambda_0$, one can choose the vectors 
$P_j^\lambda,Q_j^\lambda$, $j=1,\ldots, d$, so that they depend smoothly on $\lambda$ in a neighborhood of $\lambda_0$. We then have the following result.
\begin{lemma}\label{lem:[P,Q]}
Let $P_j^\lambda,Q_j^\lambda$, $j=1,\ldots, d$, be smooth eigenvectors in an open subset $U$ of $U$. Then 
we have
$$
	[P_j^\lambda,Q_j^\lambda] = \nabla_\lambda \eta_j(\lambda) \in \mathfrak z, \qquad j=1,\ldots, d, \quad \lambda\in \Lambda_0.
	$$
\end{lemma}
}

\begin{proof}[Proof of Lemma \ref{lem:[P,Q]}]
The differentiation of the equality  $B(\lambda)Q_j^\lambda = \eta_j(\lambda)P_j^\lambda$ with respect to $\lambda$ gives
$$
\forall \lambda'\in \mathfrak z^*
\qquad B(\lambda') Q_j^\lambda + B(\lambda)\, \lambda' \cdot \nabla_\lambda Q_j^\lambda 
=
\lambda' \cdot \nabla_\lambda \eta_j(\lambda)P_j^\lambda
+ \eta_j(\lambda)\,\lambda' \cdot \nabla_\lambda P_j^\lambda.
$$
Taking the scalar product with $P_j^\lambda$ and using $B(\lambda)^t=-B(\lambda)$ with $-B(\lambda) P_j^\lambda=\eta_j(\lambda)Q_j^\lambda$, we obtain:
$$
(B(\lambda') Q_j^\lambda,P_j^\lambda) + 
\eta_j(\lambda)( \lambda' \cdot \nabla_\lambda Q_j^\lambda ,Q_j^\lambda)
=
\lambda' \cdot \nabla_\lambda \eta_j(\lambda)(P_j^\lambda,P_j^\lambda)
+ \eta_j(\lambda)(\lambda' \cdot \nabla_\lambda P_j^\lambda,P_j^\lambda).
$$
Now,  $(Q_j^\lambda,Q_j^\lambda) = 1= (P_j^\lambda,P_j^\lambda)$.
Differentiating this  with respect to $\lambda$ 
yields $( \lambda' \cdot \nabla_\lambda Q_j^\lambda ,Q_j^\lambda)=0=
(\lambda' \cdot \nabla_\lambda P_j^\lambda,P_j^\lambda)$, and we have for all $\lambda'\in \mathfrak z^*$
$$
(B(\lambda') Q_j^\lambda,P_j^\lambda) =
\lambda' \cdot \nabla_\lambda \eta_j(\lambda).
$$
Since the left-hand side is equal to $\lambda'([Q_j^\lambda,P_j^\lambda])$ by definition of $B(\lambda')$, 
the conclusion follows. 
\end{proof}

The two lemmata above imply readily:

\begin{corollary}\label{cor:[P,Q]}
Using $\zeta=\zeta(\alpha,\lambda)=\sum_{j=1}^d\eta_j(\lambda) (2\alpha_j+1)$, we deduce that for the choice of orthonormal basis of Lemma~\ref{lem:[P,Q]}, we have 
$$
\pi^\lambda(\bP_\zeta)
 \Big(V \cdot \pi^\lambda(V))\circ T \Big)
\pi^\lambda(\bP_\zeta) =
\frac{1} 4 \mathbb L - \frac i4 \nabla_\lambda \zeta.
$$
\end{corollary}

\end{document}